\pdfoutput=1
\documentclass[final,leqno,onefignum,onetabnum,onealgnum,oneeqnum,onethmnum]{siamart171218}
\usepackage{lipsum}
\usepackage{amsfonts}
\usepackage{amsmath}
\usepackage{graphicx}
\usepackage{epstopdf}
\usepackage[caption=false]{subfig}
\graphicspath{{imgs/}}
\Crefname{Algorithm}{Algorithm}{Algorithms}
\crefformat{figure}{#2{}\scshape{Fig.} #1{}#3}
\crefformat{lemma}{#2{}\scshape{Lemma} #1{}#3}
\crefformat{theorem}{#2{}\scshape{Theorem} #1{}#3}
\crefformat{proposition}{#2{}\scshape{Proposition} #1{}#3}
\crefformat{definition}{#2{}\scshape{Definition} #1{}#3}

\crefformat{remark}{#2{}\scshape{Remark} #1{}#3}
\usepackage{algpseudocode}
\AppendGraphicsExtensions{.tif}
\usepackage{transparent}
\usepackage{tikz}
\usetikzlibrary{positioning,calc}
\usepackage{xcolor}
\definecolor{selfDefinedColor}{rgb}{0.8,0.1,0.4}
\usepackage{mwe}
\tikzset{blankboximg/.style={remember picture,white,ultra thick,draw,inner sep=0pt,outer sep=0pt}}
\tikzset{redboximg/.style={remember picture,red,ultra thick,draw,inner sep=0pt,outer sep=0pt}}
\tikzset{blueboximg/.style={remember picture,blue,ultra thick,draw,inner sep=0pt,outer sep=0pt}}
\tikzset{selfDefinedColorboximg/.style={remember picture,selfDefinedColor,ultra thick,draw,inner sep=0pt,outer sep=0pt,dashed}}
\tikzset{greenboximg/.style={remember picture,green,ultra thick,draw,inner sep=0pt,outer sep=0pt,dashed}}
\usepackage{etoolbox}
\apptocmd{\sloppy}{\hbadness 10000\relax}{}{}
\usepackage{bm}
\usepackage{url}
\usepackage[capitalize,nameinlink]{cleveref}
\title{Multidimensional TV-Stokes for image processing}
\author{
    Bin Wu\thanks{Department of Computer Science, Electrical Engineering and Mathematical Sciences, Western Norway University of Applied Sciences, Inndalsveien 28, 5063 Bergen, Norway (\email{bin.wu@hvl.no}, \email{talal.rahman@hvl.no}).}
    \and
    Xue-Cheng Tai\thanks{Department of Mathematics, University of Bergen, All\'{e}gaten 41, 5007 Bergen, Norway (\email{xue-cheng.tai@uib.no}).}
    \and
    Talal Rahman\footnotemark[1]
}
\headers{Multidimensional TV-Stokes for image processing}{B. Wu, X. C. Tai, and T. Rahman}

\begin{document}
\maketitle

\begin{abstract}
  A complete multidimential TV-Stokes model is proposed based on smoothing a gradient field in the first step and reconstruction of the multidimensional image from the gradient field. It is the correct extension of the original two dimensional TV-Stokes to multidimensions. Numerical algorithm using the Chambolle's semi-implicit dual formula is proposed. Numerical results applied to denoising 3D images and movies are presented. They show excellent performance in avoiding the staircase effect, and preserving fine structures.
\end{abstract}


\section{Introduction}
\label{intro}

As one of the major techniques in computer vision, image processing is widely used in the community. One of the important tasks of image processing is image denoising which is used for removing noises from corrupted images. It plays a crucial and distinguished role, e.g. being used as a pre-step to any processing or analysis of the images. The noise in a corrupted image may depend on the type of instrument used, the procedure applied, and the environment in which the image is taken. Noises are characterized as Gaussian or salt-pepper, additive or multiplicative \cite{chan2005}. We consider only the Gaussian noise in this paper, which is additive. 

Searching for a more general and robust algorithm to handle the problem of image denoising has been one of the major interests in the image processing community. With the pioneering work of Rudin, Osher and Fatemi (cf. \cite{Rudin1992}) based on the total variation minimization for image denoising, published in 1992, their model has been widely used as the basis for most denoising algorithms because of its elegant mathematics, simplicity of implementation, and effectiveness. Allthough the model is able to preserve edges in the image, which is important for an image, it has the limitation that it exhibits staircase effect or produces patchy images, cf. \cite{Rahman2007}. A number of algorithms now exist that are based on the total variation minimization, which try to overcome this limitation, including high order algorithms, cf. \cite{Lysaker2006,Wu2018}, the training based algorithms, cf. \cite{Chen2017}, iterative regularization algorithms, cf. \cite{Osher2005,Charest2006,Marcinkowski2016}, and algorithms based on the TV-Stokes, cf. \cite{Tai2006,Rahman2007,Litvinov2011,Hahn2012,Elo2009a,Elo2009}. There are also excellent algorithms that are not based on the variational approach, e.g., the non-local algorithms, cf. \cite{Buades2005}, the block matching algorithm, cf. \cite{Dabov2007}, and the low rank algorithm, cf. \cite{Ji2010}.

The purpose of this paper is to propose a model which can be used for the data of any dimensions, e.g. 3D medical data, 2D video, and dynamic volume data. The model is expected to be simple but effective with only a small number of parameters. The TV-Stokes model is the right candidate for this. We extend the TV-Stokes model from 2D to an arbitrary dimensions, and we call it the multi-dimensional TV-Stokes. Even though the paper addresses the task of denoising, The model can be extended for use in other data restoration tasks, like image inpainting, image compression, etc. .

The structure of the paper is organized as follows. In Section 2, we present the first step of the multi-dimensional TV-Stokes and its dual formulation. is addressed accompanying with related analysis. In the Section 3, Chambolle's semi-implicit iteration is presented for the minimization. The key part of the multi-dimensional TV-Stokes is the orthogonal projection onto the constrained spaces, which is presented in Section 4, based on the singular value decomposition of the finite difference matrix $D$. Section 5 gives the second step of the proposed model, height map reconstruction using a surface fitting technique. Numerical results are presented in Section 6, comparing the proposed model with the Rudin-Osher-Fatemi (ROF) model.

\section{The vector field smoothing step of the TV-Stokes}
\label{sec:2}

The first step of the TV-Stokes is the vector field smoothing step. Its required that the desired smoothed vector field is a gradient field, that is $\bm{g} = \nabla u$ where $\bm{g} \in \mathbb{R}^{N_1 \times \ldots N_d \times d}$ and $u \in \mathbb{R}^{N_1 \times \ldots N_d}$ is a $d$-dimensional grayscale image. In the case $d = 2$, it is the same as requiring $\nabla \times \bm{g} = 0$ for the gradient field, which is equivalent to $\nabla \cdot \bm{g}^{\bot} = 0$, which is exactly the constraint used in the TV-Stokes model for 2D.

{\sloppy Let $u \in \mathbb{R}^{N_1 \times \cdots \times N_d}$ denote the components of a $d$-dimensional grayscale image with components $u_{i_1 \ldots i_d}, i_j=1,\ldots,N_j$ for $j=1,\ldots,d$. When a linear transformation defined by a matrix $T \in \mathbb{R}^{N_i \times N_i}$ is applied to the $i$-th dimension of $u$, its result will be denoted by $T_{\times i}d$. For instance,
  \[ (T_{\times2}u)_{ji_1i_3} = \sum_{i_2} T_{ji_2} u_{i_1i_2i_3}.\]
} The gradient field of $u$ is $\bm{g} = \left( g_1,g_2, \ldots, g_d \right)= \nabla{u} \in \mathbb{R}^{N_1 \times \cdots \times N_d \times d}$, where $\nabla$ signifies a discrete gradient operator with homogeneous Neumann boundary conditions. Namely, the gradient operator $\nabla u = \left( D_{\times1} u, D_{\times2} u, \ldots, D_{\times d} u \right)$ is determined by means of the following $N_i \times N_i$ differentiation matrix
\begin{equation}\label{eqs:D}
  D = \left[\begin{array}{ccccc}
      -1 & 1                        \\
         & -1 & 1                   \\
         &    & \ddots & \ddots     \\
         &    &        & -1     & 1 \\
         &    &        & 0      & 0
    \end{array}\right].
\end{equation}

Given a $d$-dimensional image $u^o \in \mathbb{R}^{N_1 \times \cdots \times N_d}$ corrupted by a Gaussian noise, a constrained variant of the Rudin-Osher-Fatemi model \cite{Rudin1992} consists in finding a gradient field $\bm{g} = \left( g_1, \ldots, g_d \right) \in \mathbb{R}^{N_1 \times \cdots \times N_d \times d}$ satisfying the minimization problem
\begin{equation}\label{eqs:TVS}
  \min_{\begin{subarray}{c} \bm{g} = \nabla u \\ u \in \mathbb{R}^{N_1 \times \cdots \times N_d} \end{subarray}}
  \left[\| \nabla \bm{g} \|_1 + \cfrac{1}{2\lambda} \| \bm{g} - {\bm{g}}^o \|_2^2 \right],
\end{equation}
where $\lambda > 0$ is a suitable scalar parameter, ${\bm{g}}^o = \nabla u^o$ and
\begin{equation*}
  h = \nabla \bm{g} =
  \left( \begin{array}{cccc}
      D_{\times1}g_1 & D_{\times2}g_1 & \cdots & D_{\times d}g_1 \\
      D_{\times1}g_2 & D_{\times2}g_2 & \cdots & D_{\times d}g_2 \\
      \vdots         & \vdots         & \cdots & \vdots          \\
      D_{\times1}g_d & D_{\times2}g_d & \cdots & D_{\times d}g_d
    \end{array} \right) \in \mathbb{R}^{N_1 \times \cdots \times N_d \times d \times d}.
\end{equation*}

The (isotropic) norms in (\ref{eqs:TVS}) are as follows:
\[ \begin{array}{l}
    \| \bm{g} \|_2 = \sqrt{ \sum_{i_1 \ldots i_d} g_{1, i_1 \ldots i_d}^2 + g_{2, i_1 \ldots i_d}^2 + \cdots + g_{d, i_1 \ldots i_d}^2 }, \\
    \| h \|_1 = \sum_{i_1 \ldots i_d} \sqrt{ \sum_{l, m=1}^d h_{lm, i_1 \ldots i_d}^2 }.
  \end{array} \]
The linear subspace $V = \{ \bm{g} = \nabla u \colon u \in \mathbb{R}^{N_1 \times \cdots \times N_d} \} \subset \mathbb{R}^{N_1 \times \cdots \times N_d \times d}$ is the range of the orthogonal projector
\begin{equation}\label{eqs:proj}
  \Pi = \nabla(\nabla^* \nabla)^{\dagger} \nabla^*,
\end{equation}
where $\nabla^*$ is the adjoint to the $\nabla$ operator, and $\dagger$ denotes the Moore-Penrose pseudoinverse. By the way, ${\bm{g}}^o = \Pi {\bm{g}}^o \in V$.

The continuous functional
\[ \| \nabla \bm{g} \|_1 + \cfrac{1}{2\lambda} \| \bm{g} - {\bm{g}}^o \|_2^2, \quad \bm{g} \in V, \]
is strictly convex, its minimum is unique and attained in the closed ball $\{ \bm{g} \colon \| \bm{g} - {\bm{g}}^o \|_2 \leq \| {\bm{g}}^o \|_2 \}$.

By the aid of a dual variable
\[p = \left( \begin{array}{cccc}
      p_{11} & p_{12} & \cdots & p_{1d} \\
      p_{21} & p_{22} & \cdots & p_{2d} \\
      \vdots & \vdots & \cdots & \vdots \\
      p_{d1} & p_{d2} & \cdots & p_{dd}
    \end{array} \right) \in \mathbb{R}^{N_1 \times \cdots \times N_d \times d \times d}\]
such that $\| \nabla \bm{g} \|_{1} = \max \limits_{\Vert p \Vert_{\infty} \leq 1} \langle \nabla \bm{g}, p \rangle = \max \limits_{\Vert p \Vert_{\infty} \leq 1} \langle \bm{g}, \nabla^* p \rangle$ the model (\ref{eqs:TVS}) reduces to the equivalent min-max formulation
\begin{equation}\label{eqs:minmax}
  \min_{\bm{g }\in V} \max_{\Vert p \Vert_{\infty} \leq 1} F(\bm{g},p),
  \quad \mbox{where} \quad
  F(\bm{g}, p) = \langle \bm{g}, \nabla^* p \rangle + \cfrac{1}{2\lambda} \langle \bm{g} - {\bm{g}}^o, \bm{g} - {\bm{g}}^o \rangle.
\end{equation}
Here the max norm is $\Vert p \Vert_{\infty} = \max_{i_1\ldots i_d} \sqrt{ \sum_{l,m=1}^d |p_{lm,i_1 \ldots i_d}|^2 }$.

The generalized minimax theorem, cf. 
\cite{Sion1957}
, justifies the equality
\begin{equation*}
  \min_{\bm{g} \in V} \max_{|p| \leq 1} F(\bm{g}, p) = \max_{|p| \leq 1} \min_{\bm{g} \in V} F(\bm{g}, p),
\end{equation*} so that the model (\ref{eqs:TVS}) reduces to the max-min formulation
\begin{equation}\label{eqs:maxmin}
  \max_{\Vert p \Vert_{\infty} \leq 1} \min_{\bm{g} \in V} \langle \bm{g}, \nabla^* p \rangle +
  \cfrac{1}{2\lambda} \langle \bm{g} - {\bm{g}}^o, \bm{g} - {\bm{g}}^o \rangle.
\end{equation}
In spite of a possible nonuniqueness of $p$, the solution $\bm{g}$ obtained by means of (\ref{eqs:maxmin}) 
is unique. Indeed, if $(\bm{g}_*, p_*)$ and $(\bm{g}, p)$ are solutions of 
(\ref{eqs:minmax}) 
and (\ref{eqs:maxmin}) respectively, then $F(\bm{g}_*, p) = F(\bm{g}, p)$ by Lemma 36.2 in \cite{Rockafellar1997}. The  strict convexity of the mapping $\bm{g} \mapsto F(\bm{g}, p)$ implies $\bm{g} = \bm{g}_*$.

Since the constraint $\bm{g} \in V$ is equivalent to the condition $\Pi \bm{g} = \bm{g}$, the term $\langle \bm{g}, \nabla^*p \rangle$ in (\ref{eqs:maxmin}) can be replaced by $\langle \bm{g}, \nabla^* p \rangle = \langle \bm{g}, \Pi \nabla^* p \rangle$. Now we observe that the unique solution
\begin{equation}\label{eqs:umin}
  \bm{g} = {\bm{g}}^o - \lambda \Pi \nabla^* p
\end{equation} of the unconstrained problem $\min_{\bm{g} \in \mathbb{R}^{N_1 \times \cdots \times N_d \times d}} \langle \bm{g}, \Pi \nabla^* p \rangle + \cfrac{1}{2\lambda} \langle \bm{g} - {\bm{g}}^o, \bm{g} - {\bm{g}}^o \rangle$ also solves the constrained problem $\min_{\bm{g} \in V} \langle \bm{g}, \nabla^* p \rangle + \cfrac{1}{2\lambda} \langle \bm{g} - {\bm{g}}^o, \bm{g} - {\bm{g}}^o \rangle$.

Under the condition (\ref{eqs:umin})
\begin{equation*}
  F(\bm{g}, p) = \cfrac{1}{2\lambda} \| {\bm{g}}^o \|_2^2 - \cfrac{1}{2\lambda} \| {\bm{g}}^o - \lambda\Pi\nabla^*p \|_2^2,
\end{equation*} and we arrive at the minimum distance problem
\begin{equation}\label{eqs:dist}
  \min_{\Vert p \Vert_{\infty} \leq 1} \| {\bm{g}}^o - \lambda \Pi \nabla^* p \|_2.
\end{equation}

\subsection{Chambolle's iteration}
\label{sec:3}

We solve the minimization problem (\ref{eqs:dist}) using Chambolle's iteration, cf. \cite{Chambolle2004,Elo2009a}. Accordingly, the Karush-Kuhn-Tucker conditions for (\ref{eqs:dist}) are given by the equation
\begin{equation}\label{eqs:KKT}
  \nabla \left( \Pi \nabla^* p - {\bm{g}}^o / \lambda \right) + | \nabla \left( \Pi \nabla^* p - {\bm{g}}^o / \lambda \right) | \cdot p = 0,
\end{equation} where the dot signifies the entrywise product.
Solution of (\ref{eqs:dist}) is approximated by the semi-implicit iteration
\begin{equation}\label{eqs:iter}
  \begin{array}{l}
    p^0 = 0, \\
    p^{k+1} = \cfrac{
      p^k - \tau \nabla \left( \Pi \nabla^* p^k - \cfrac{{\bm{g}}^o}{\lambda} \right)
    }{
      \max \left\{ 1, \left| p^k - \tau \nabla \left( \Pi \nabla^* p^k - \cfrac{{\bm{g}}^o}{\lambda} \right) \right| \right\}
    },
  \end{array}
\end{equation} where the maximum and division are both entrywise operations.
\begin{lemma}
  The semi-implicit iteration (\ref{eqs:iter}) is 1-Lipschitz \cite{heinonen2012} for
  \begin{equation}\label{eqs:tau}
    \tau \leq \cfrac{2}{\| \nabla \Pi \nabla^* \|_2}.
  \end{equation}
\end{lemma}
\begin{proof}
  Each step of (\ref{eqs:iter}) uses two mappings: $p\mapsto p - \tau \nabla \left( \Pi \nabla^* p - {\bm{g}}^o / \lambda \right)$ and $q \mapsto q / \max(1, |q|)$. The first mapping is linear and 1-Lipschitz if and only if $\| I - \tau \nabla \Pi \nabla^* \|_2 \leq 1$, where $I$ is the identity transformation. The second mapping is always 1-Lipschitz.\qed
\end{proof}
The SVD of the matrix $D$ allows us to show that 1-Lipschitz holds when
\begin{equation}\label{eqs:tau1}
  \tau \leq \cfrac{1}{2d}.
\end{equation}

\subsection{Implementation of \texorpdfstring{$\Pi = \nabla(\nabla^* \nabla)^{\dagger} \nabla^*$}{Pi operator}}
\label{sec:4}

The intricate part in our algorithm is the implementation of the $\Pi$ operation. 
We describe here how this can be done efficiently. The singular value decomposition of the differentiation matrix $D$ is
\begin{equation*}
  D = \left[
    \begin{array}{cc}
      0 & S \\
      1 & 0
    \end{array}
    \right]
  \Sigma C,
\end{equation*} where the diagonal $\Sigma$ has the entries $\Sigma_{ii} = 2 \sin \cfrac{\pi (i-1)}{2N}$, $i=1, \ldots, N$.
The orthogonal $N\times N$ matrix $C$ of the discrete cosine transform has the entries
\begin{equation*}
  \begin{array}{l}
    C_{1j} = \sqrt{\cfrac{1}{N}}, \\
    C_{ij} = \sqrt{\cfrac{2}{N}} \cos \cfrac{\pi (i-1)(2j-1)}{2N}, \quad i=2, \ldots, N, \quad j=1, \ldots, N.
  \end{array}
\end{equation*}
The orthogonal $(N-1)\times (N-1)$ matrix $S$ of the discrete sine transform has the entries
\begin{equation*}
  S_{ij} = \sqrt{\cfrac{2}{N}} \sin \cfrac{\pi ij}{N},\quad i, j=1, \ldots, N-1.
\end{equation*}
We can use the singular value decomposition of $D$ to compute $\| \nabla \|_2$. Indeed, $\| \nabla \|_2 = \| \nabla \|_2 = 2\| \left[ \sin \cfrac{\pi (N_1-1)}{2N_1}, \ldots, \sin \cfrac{\pi (N_d-1)}{2N_d} \right] \|_2 \approx 2 \sqrt{d}$.

{\sloppy We implement $(\nabla^* \nabla)^{\dagger}$ by the fast Fourier transform.
Since $\nabla^* \nabla u = (D^TD)_{\times 1} u + (D^TD)_{\times 2} u + \ldots + (D^TD)_{\times d} u,$ the equation $\nabla^* \nabla u = f$ is equivalent to
\begin{equation*}
  \Sigma^2_{\times 1} \widehat{u} + \Sigma^2_{\times 2} \widehat{u} + \ldots + \Sigma^2_{\times d} \widehat{u} = \widehat{f},
\end{equation*} where $\widehat{u} = C_{\times 1} C_{\times 2} \ldots  C_{\times d} u$ and $\widehat{f} = C_{\times 1} C_{\times 2} \ldots C_{\times d} f$. The components of $\widehat{u}$ and $\widehat{f}$ are related by the equalities
\begin{equation*}
  \widehat{u}_{i_1 \cdots i_d} (\Sigma_{i_1 i_1} + \Sigma_{i_2 i_2} + \ldots + \Sigma_{i_d i_d}) = \widehat{f}_{i_1 \cdots i_d}, \quad i_k = 1, \ldots, N_k.
\end{equation*} Note that $\Sigma_{i_1 i_1} + \Sigma_{i_2 i_2} + \ldots + \Sigma_{i_d i_d}$ equals zero if $i_1 = \ldots = i_d = 1$ and is positive otherwise. Hence
\begin{equation*}
  \displaystyle
  \widehat{u}_{1 \cdots 1} = 0,\quad
  \widehat{u}_{i_1 \cdots i_d} = \widehat{f}_{i_1 \cdots i_d} / (\Sigma_{i_1 i_1} + \Sigma_{i_2 i_2} + \ldots + \Sigma_{i_d i_d}) \quad \mbox{if } i_1 + \ldots + i_d > d.
\end{equation*} Recall that $u = C^T_{\times 1}  C^T_{\times 2}  \cdots C^T_{\times d} \widehat{u}$. Moreover, multiplication with the matrices $C$ and $C^T$ should be carried out by the fast Fourier transform FFT.}

\section{The scalar field reconstruction step of the TV-Stokes}
\label{sec:5}

We extend the second step or the image reconstruction step of the original TV-Stokes \cite{Rahman2007} for 2D to multidimensions. Accordingly, for the given $d$-dimensional image $u^o \in \mathbb{R}^{N_1 \times \cdots \times N_d}$,  which is corrupted by a Gaussian noise, a vector matching model \cite{Rahman2007} is applied which satisfies the minimization problem,
\begin{equation}\label{eqs:TVmatch}
  \min_{
    \begin{subarray}{c}
      u \in \mathbb{R}^{N_1 \times \cdots \times N_d}
    \end{subarray}
  }
  \left[
    \| \nabla u \|_1 + \cfrac{1}{2\lambda} \| u - u^o \|_2^2 - \nabla u \cdot \cfrac{\bm{g}}{|\bm{g}|}
    \right],
\end{equation} where $\lambda > 0$ is a suitable scalar parameter, $\nabla u = \left( D_{\times1} u, D_{\times2} u, \ldots, D_{\times d} u \right)$ and
\begin{equation*}
  \| \nabla u \|_1 =: \| s \|_1 = \sum_{i_1 \ldots i_d} \sqrt{ \sum_{l=1}^d s_{l, i_1 \ldots i_d}^2 }.
\end{equation*} Note that the $\bm{g}$ is obtained from the first step and thus, in the reconstruction, it is a known term. By completing the quadratic term, the above problem (\ref{eqs:TVmatch}) can be reformed as
\begin{equation}\label{eqs:TVmatchR}
  \min_{
    \begin{subarray}{c}
      u \in \mathbb{R}^{N_1 \times \cdots \times N_d}
    \end{subarray}
  }
  \left[
    \| \nabla u \|_1 + \cfrac{1}{2\lambda} \| u + \lambda \nabla \cdot \cfrac{\bm{g}}{|\bm{g}|} - u^o \|_2^2
    \right],
\end{equation} which is strictly convex, its minimum is unique and attained in the closed ball $\Big \{ u \colon \Big \| u - (u^o - \lambda \nabla \cdot \cfrac{\bm{g}}{|\bm{g}|}) \Big \|_2 \leq \Big \| u^o - \lambda \nabla \cdot \cfrac{\bm{g}}{|\bm{g}|} \Big \|_2 \Big \}$.

With the aid of a dual variable $p = (p_{1}, p_{2}, \cdots, p_{d})$ such that $\| \nabla u \|_{1} $ $= \max \limits_{\Vert p \Vert_{\infty} \leq 1} \langle \nabla u,$ $p \rangle $ $= \max \limits_{\Vert p \Vert_{\infty} \leq 1} \langle u, \nabla^* p \rangle$ the model (\ref{eqs:TVmatch}) reduces to the equivalent min-max formulation
\begin{equation}\label{eqs:minmaxmatch}
  \begin{array}{l}
    \min_{
      u \in \mathbb{R}^{N_1 \times \cdots \times N_d}
    }
    \max_{
    \Vert p \Vert_{\infty} \leq 1
    }
    F(u,p), \\
    \mbox{where} \quad F(u, p) = \langle u, \nabla^* p \rangle + \cfrac{1}{2\lambda} \| u - u^o \|_2^2 - \nabla u \cdot \cfrac{\bm{g}}{|\bm{g}|}.
  \end{array}
\end{equation} Here the max norm is $\Vert p \Vert_{\infty} = \max_{i_1 \ldots i_d} \sqrt{ \sum_{l=1}^d |p_{l, i_1 \ldots i_d}|^2 }$.

The generalized minimax theorem, cf. \cite{Sion1957}, justifies the equality
\begin{equation*}
  \min_{
    u \in \mathbb{R}^{N_1 \times \cdots \times N_d}
  }
  \max_{
  |p|_{\infty} \leq 1
  }
  F(u,p)
  =
  \max_{
  |p|_{\infty} \leq 1
  }
  \min_{
    u \in \mathbb{R}^{N_1 \times \cdots \times N_d}
  }
  F(u,p),
\end{equation*} so that the model (\ref{eqs:TVmatch}) reduces to the max-min formulation
\begin{equation}\label{eqs:maxminmatch}
  \max_{
  \Vert p \Vert_{\infty} \leq 1
  }
  \min_{
    u \in \mathbb{R}^{N_1 \times \cdots \times N_d}
  }
  \langle u, \nabla^* p \rangle + \cfrac{1}{2 \lambda} \| u - u^o \|_2^2 - \nabla u \cdot \cfrac{\bm{g}}{|\bm{g}|}.
\end{equation} In spite of a possible nonuniqueness of $p$, the solution $u$ obtained by means of (\ref{eqs:maxminmatch}) is unique. Indeed, if $(u_*, p_*)$ and $(u, p)$ are solutions of (\ref{eqs:minmaxmatch}) and (\ref{eqs:maxminmatch}) respectively, then $F(u_*, p)=F(u, p)$ by (Lemma 36.2 of \cite{Rockafellar1997}). The strict convexity of the mapping $u\mapsto F(u, p)$ implies $u = u_*$. The unique solution is obtained as follows
\begin{equation}\label{eqs:uminmatch}
  u = u^o - \lambda (\nabla^*p + \nabla \cdot \cfrac{\bm{g}}{|\bm{g}|})
\end{equation} Under the condition (\ref{eqs:uminmatch})
\begin{equation*}
  F(u,p)
  =
  \cfrac{1}{2\lambda} \Big \| u^o - \lambda \nabla \cdot \cfrac{\bm{g}}{|\bm{g}|} \Big \|_2^2
  -
  \cfrac{1}{2 \lambda} \Big \| u^o - \lambda \nabla \cdot \cfrac{\bm{g}}{|\bm{g}|} -\lambda \nabla^* p \Big \|_2^2,
\end{equation*} and we arrive at the minimum distance problem
\begin{equation}\label{eqs:distmatch}
  \min_{
  \Vert p \Vert_{\infty} \leq 1
  }
  \Big \| u^o - \lambda \nabla \cdot \cfrac{\bm{g}}{|\bm{g}|} - \lambda \nabla^* p \Big \|_2.
\end{equation}

As in \cite{Chambolle2004,Elo2009a} the Karush-Kuhn-Tucker conditions for (\ref{eqs:distmatch}) are given by the equation
\begin{equation}\label{eqs:KKTmatch}
  \nabla \left( - \cfrac{u^o}{\lambda} + \nabla \cdot \cfrac{\bm{g}}{|\bm{g}|} + \nabla^* p \right) + \left | \nabla \left( - \cfrac{u^o}{\lambda} + \nabla \cdot \cfrac{\bm{g}}{|\bm{g}|} + \nabla^* p \right) \right | \cdot p = 0,
\end{equation} where the dot signifies the entrywise product. Solution of (\ref{eqs:distmatch}) is approximated by the semi-implicit iteration
\begin{equation}\label{eqs:itermatch}
  \begin{array}{l}
    p^0 \quad = 0, \\
    p^{k+1} = \cfrac{
      p^k - \tau \nabla \left( \nabla \cdot \cfrac{\bm{g}}{|\bm{g}|}+ \nabla^* p - \cfrac{u^o}{\lambda} \right)
    }{
      \max \left\{ 1, \left| p^k - \tau \nabla \left( \nabla \cdot \cfrac{\bm{g}}{|\bm{g}|} + \nabla^* p - \cfrac{u^o}{\lambda} \right) \right| \right\}
    },
  \end{array}
\end{equation} where the maximum and division are entrywise operations.

\begin{lemma}
  The semi-implicit iteration (\ref{eqs:itermatch}) is 1-Lipschitz for
  \begin{equation}\label{eqs:taumatch}
    \tau \leq \cfrac{2}{\| \nabla \nabla^* \|_2}.
  \end{equation}
\end{lemma}
\begin{proof}
  Each step of (\ref{eqs:itermatch}) uses two mappings:
  \begin{equation*}
    p \mapsto p - \tau \nabla \left( \nabla \cdot \cfrac{\bm{g}}{|\bm{g}|} + \nabla^* p -\cfrac{u^o}{\lambda} \right)
  \end{equation*} and $q \mapsto q / \max (1, |q|)$. The first mapping is linear and 1-Lipschitz if and only if $\| I - \tau \nabla \nabla^* \|_2 \leq 1$, where $I$ is the identity transformation. The second mapping is always 1-Lipschitz.\qed
\end{proof}
The SVD of the matrix $D$ allows us to show that 1-Lipschitz holds when
\begin{equation}
  \tau \leq \cfrac{1}{2d}.
  \label{eqs:tau1match}
\end{equation}

\section{Numerical experiments}
\label{sec:6}

We first test the proposed model on some 3d static data, which are computer tomography data from two lung tissues of human. In \cref{fig:slicesa}, \cref{fig:slicesb1} and \cref{fig:slicesb2}, 3 groups of orthogonal slices are presented, with respect to $xz$-plane, $xy$-plane, and $yz$-plane. In these figures, the last rows show relative positions of those slices in their embedded volume data. In the first column of each figures, the 3-slice view is taken on the raw data, which contains some unknown noise. The noise type is supposed to be Gaussian. In the middle column and the last column of each figures, comparison is presented between the denoised data of using ROF model and the proposed multidimensional TV-Stokes model. The one (last column) using the proposed multidimensional TV-Stokes model results higher fidelity to the raw data than the one (middle column) using ROF model and preserves more fine structures as well as better continuity of tube-like structures. \cref{fig:isosurfacea} and \cref{fig:isosurfaceb} show the isosurfaces of the volume data. In these figures, the results from proposed model (last column) well trade-off the smoothness and details preservation. It observes that the fine tube-like structures disappeared while the ROF model smooths the noisy data.

\begin{figure}[!htbp]
  \captionsetup[subfigure]{justification=centering}
  \centering
  \begin{minipage}{1.0\textwidth}
    \centering
    \subfloat
    {
      \includegraphics[width=0.3\linewidth,
        angle=0]{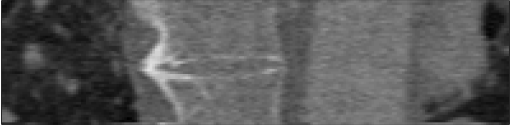}
    }
    \subfloat
    {
      \includegraphics[width=0.3\linewidth,
        angle=0]{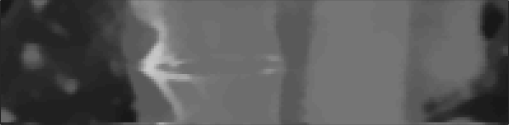}
    }
    \subfloat
    {
      \includegraphics[width=0.3\linewidth,
        angle=0]{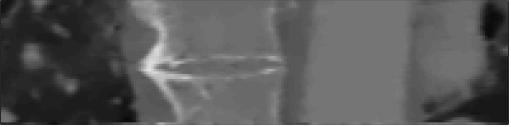}
    }
    \vspace{-0.02\linewidth}
  \end{minipage}
  \begin{minipage}{1.0\textwidth}
    \centering
    \subfloat
    {
      \includegraphics[width=0.3\linewidth,
        angle=0]{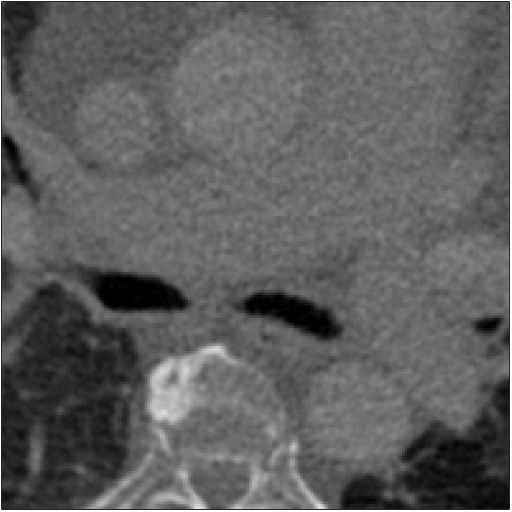}
    }
    \subfloat
    {
      \includegraphics[width=0.3\linewidth,
        angle=0]{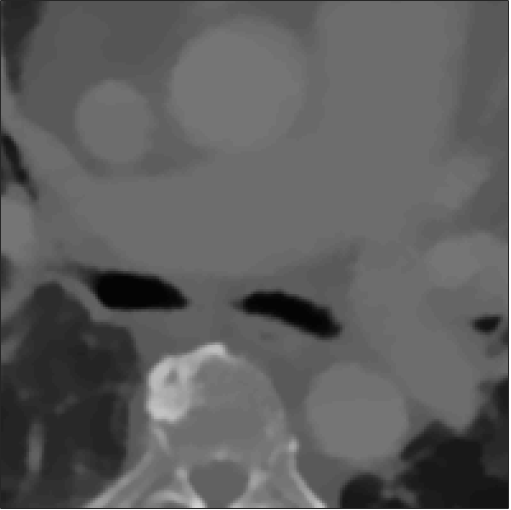}
    }
    \subfloat
    {
      \includegraphics[width=0.3\linewidth,
        angle=0]{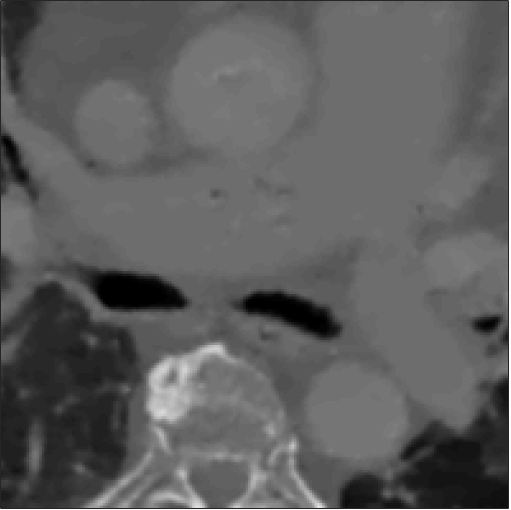}
    }
    \vspace{-0.02\linewidth}
  \end{minipage}
  \begin{minipage}{1.0\textwidth}
    \centering
    \subfloat
    {
      \includegraphics[width=0.3\linewidth,
        angle=0]{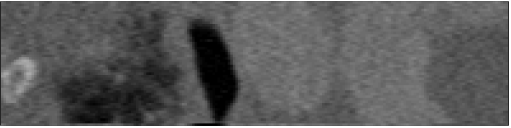}
    }
    \subfloat
    {
      \includegraphics[width=0.3\linewidth,
        angle=0]{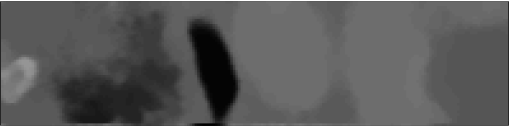}
    }
    \subfloat
    {
      \includegraphics[width=0.3\linewidth,
        angle=0]{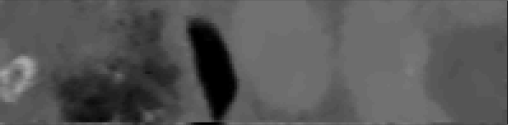}
    }
  \end{minipage}
  \caption{Visualization of computer tomography data of human lung tissue A via 3 groups of orthogonal slices. Row one, row two and row three present respectively the $xz$-slices, the $xy$-slices and the $yz$-slices. Column (a) visualizes the raw data with some noise (Gaussian noise). Column (b) visualizes the data denoised with ROF model. Column (c) visualizes the data denoised with proposed multidimensional TV-Stokes model.}
  \label{fig:slicesa}
\end{figure}

\begin{figure}[!htbp]
  \captionsetup[subfigure]{justification=centering}
  \centering
  \begin{minipage}{1.0\textwidth}
    \centering
    \subfloat
    {
      \includegraphics[width=0.3\linewidth,
        angle=0]{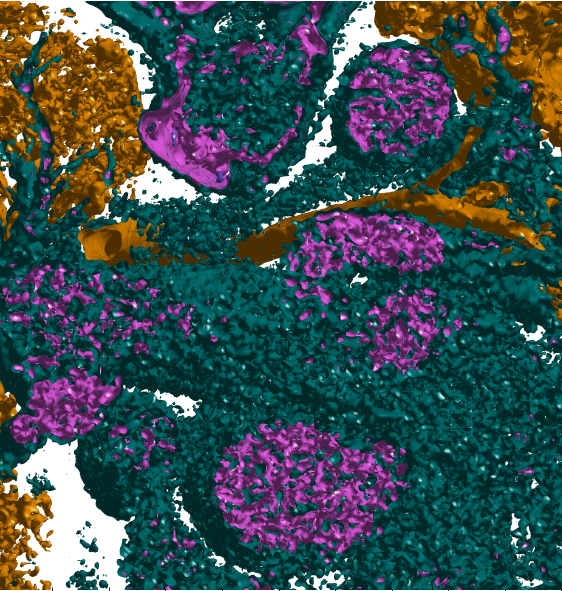}
    }
    \subfloat
    {
      \includegraphics[width=0.3\linewidth,
        angle=0]{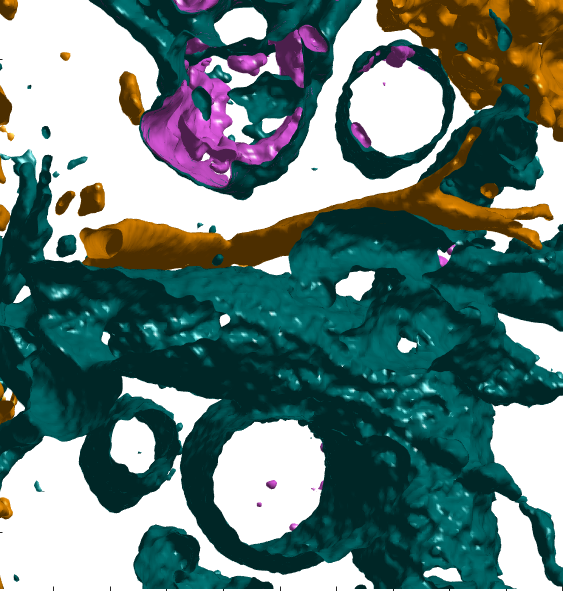}
    }
    \subfloat
    {
      \includegraphics[width=0.3\linewidth,
        angle=0]{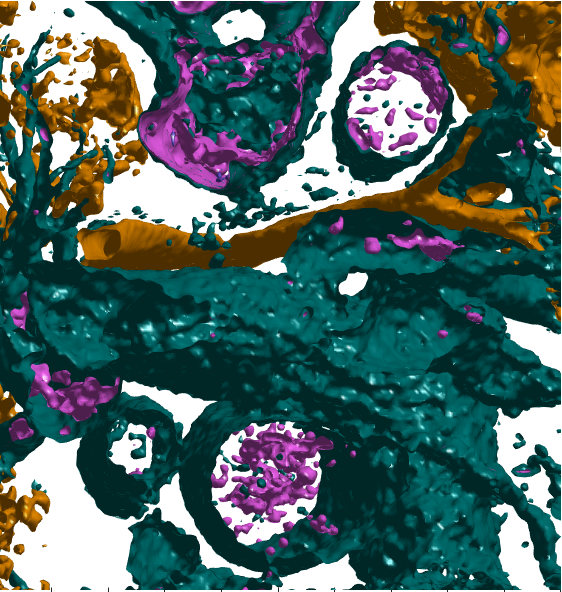}
    }
  \end{minipage}
  \caption{Isosurface show of computer tomography data of human lung tissue A. (a) raw data with some noise (Gaussian noise). (b) data denoised with ROF model. (c) data denoised with proposed multidimensional TV-Stokes model.}
  \label{fig:isosurfacea}
\end{figure}

\begin{figure}[!htbp]
  \captionsetup[subfigure]{justification=centering}
  \centering
  \begin{minipage}{1.0\textwidth}
    \centering
    \subfloat
    {
      \includegraphics[width=0.3\linewidth,
        angle=0]{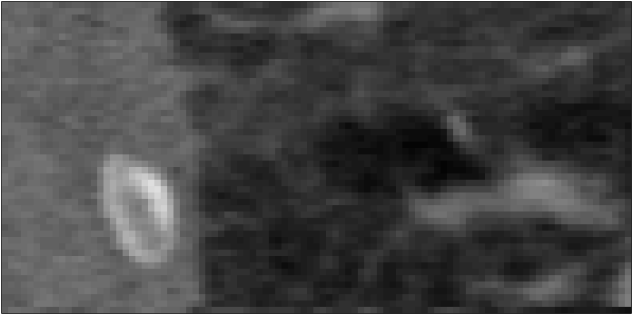}
    }
    \subfloat
    {
      \includegraphics[width=0.3\linewidth,
        angle=0]{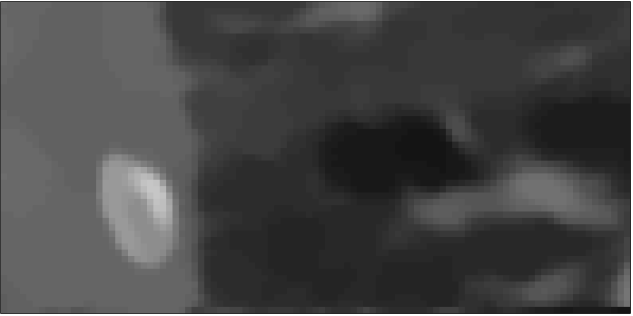}
    }
    \subfloat
    {
      \includegraphics[width=0.3\linewidth,
        angle=0]{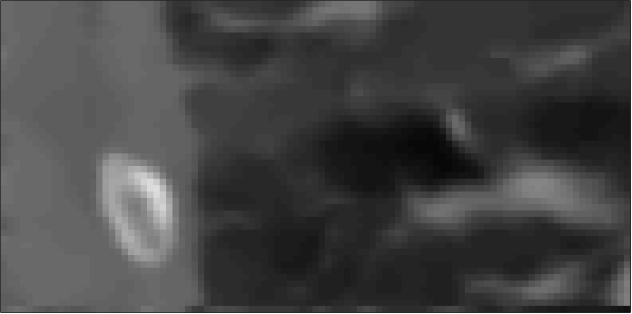}
    }
    \vspace{-0.02\linewidth}
  \end{minipage}
  \begin{minipage}{1.0\textwidth}
    \centering
    \subfloat
    {
      \includegraphics[width=0.3\linewidth,
        angle=0]{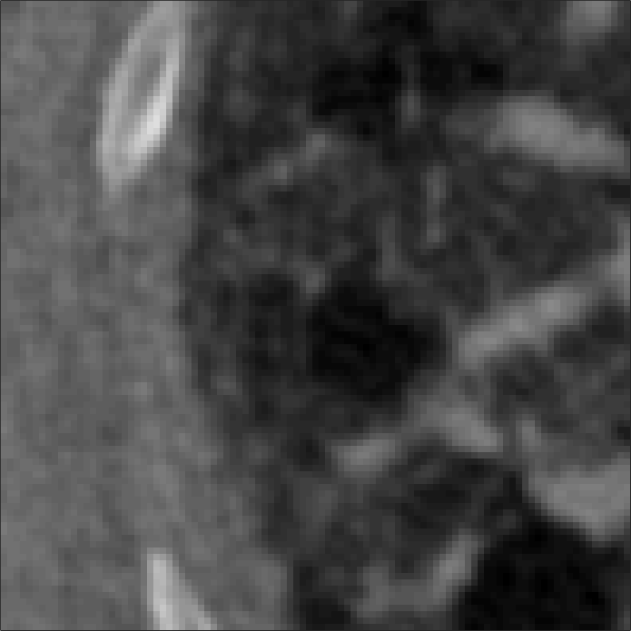}
    }
    \subfloat
    {
      \includegraphics[width=0.3\linewidth,
        angle=0]{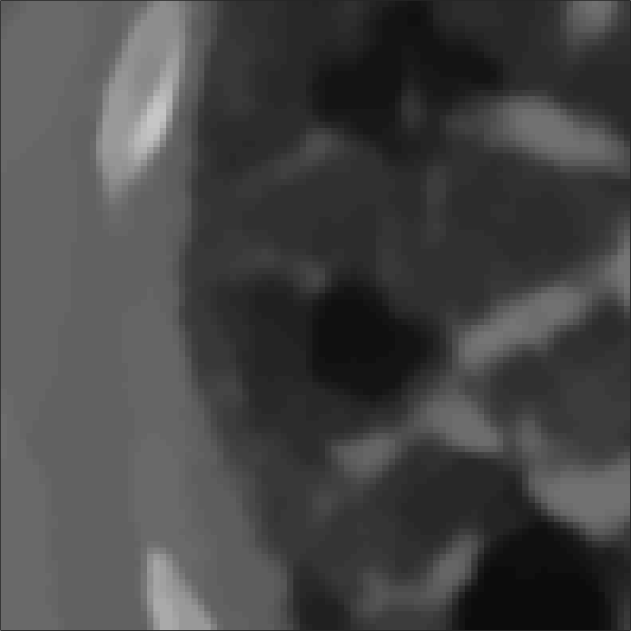}
    }
    \subfloat
    {
      \includegraphics[width=0.3\linewidth,
        angle=0]{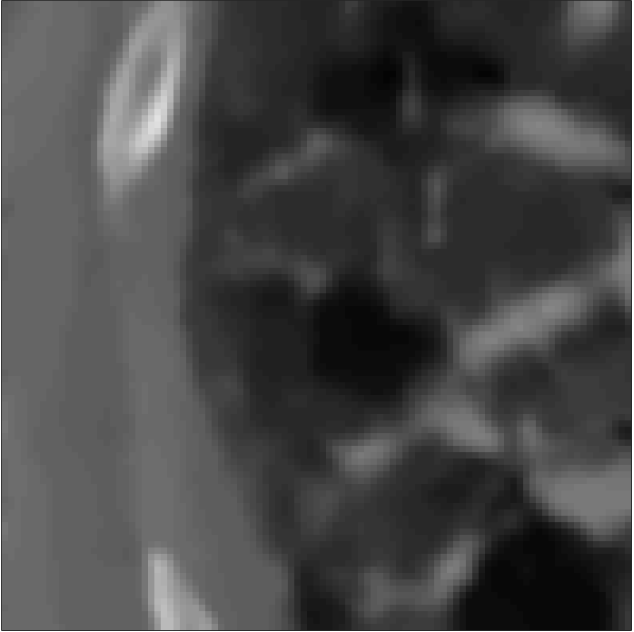}
    }
    \vspace{-0.02\linewidth}
  \end{minipage}
  \begin{minipage}{1.0\textwidth}
    \centering
    \subfloat
    {
      \includegraphics[width=0.3\linewidth,
        angle=0]{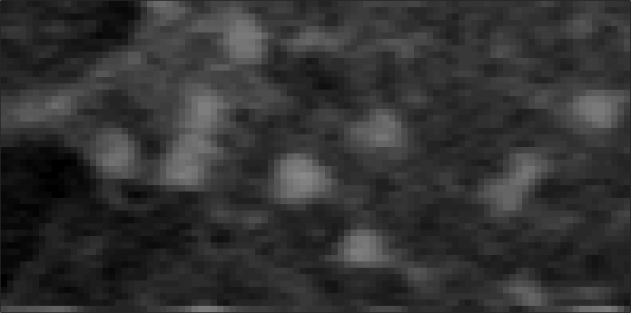}
    }
    \subfloat
    {
      \includegraphics[width=0.3\linewidth,
        angle=0]{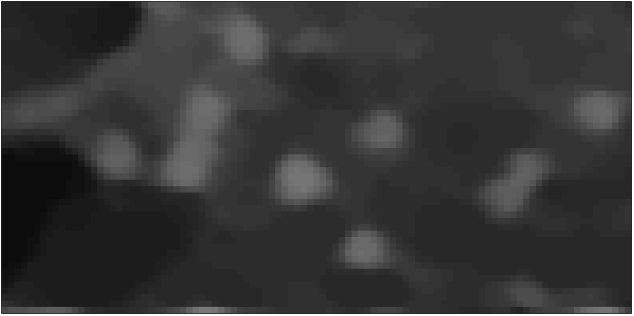}
    }
    \subfloat
    {
      \includegraphics[width=0.3\linewidth,
        angle=0]{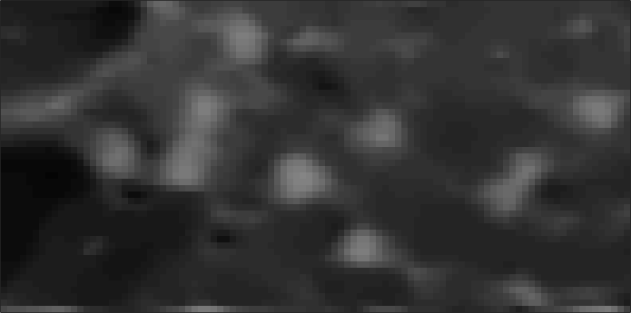}
    }
  \end{minipage}
  \caption{Visualization of computer tomography data of human lung tissue B via 3 groups of orthogonal slices. Row one, row two and row three present respectively the $xz$-slices, the $xy$-slices and the $y$-slices. Column (a) visualizes the raw data with some noise (Gaussian noise). Column (b) visualizes the data denoised with ROF model. Column (c) visualizes the data denoised with proposed multidimensional TV-Stokes model.}
  \label{fig:slicesb1}
\end{figure}

\begin{figure}[!htbp]
  \captionsetup[subfigure]{justification=centering}
  \centering
  \begin{minipage}{1.0\textwidth}
    \centering
    \subfloat
    {
      \includegraphics[width=0.3\linewidth,
        angle=0]{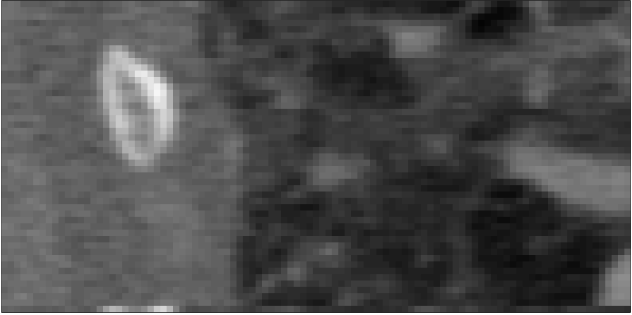}
    }
    \subfloat
    {
      \includegraphics[width=0.3\linewidth,
        angle=0]{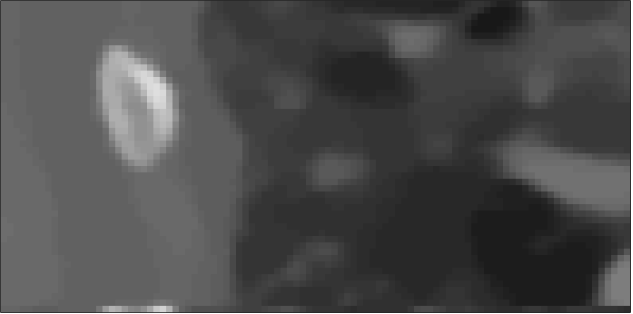}
    }
    \subfloat
    {
      \includegraphics[width=0.3\linewidth,
        angle=0]{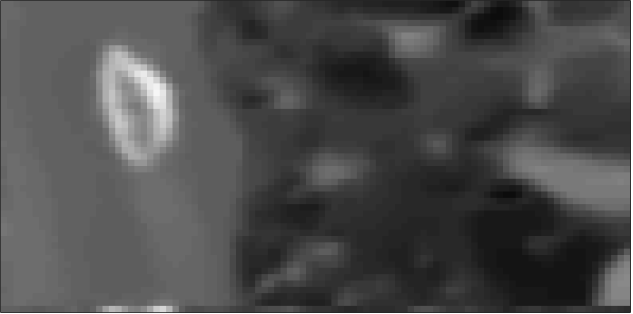}
    }
    \vspace{-0.02\linewidth}
  \end{minipage}
  \begin{minipage}{1.0\textwidth}
    \centering
    \subfloat
    {
      \includegraphics[width=0.3\linewidth,
        angle=0]{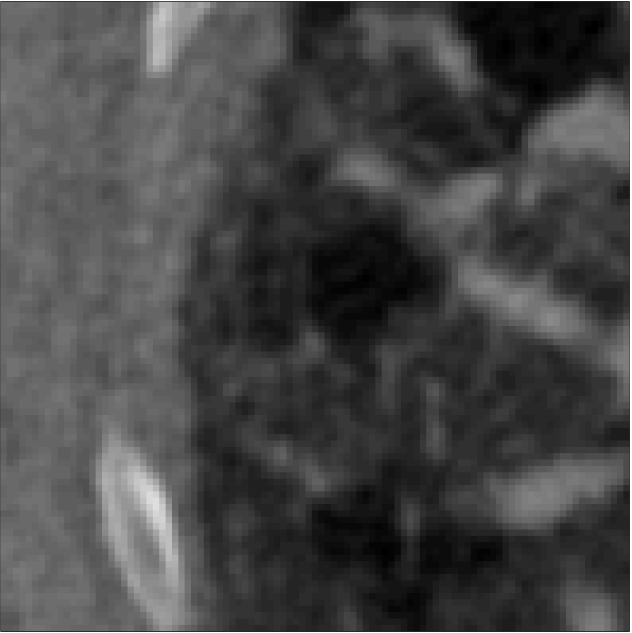}
    }
    \subfloat
    {
      \includegraphics[width=0.3\linewidth,
        angle=0]{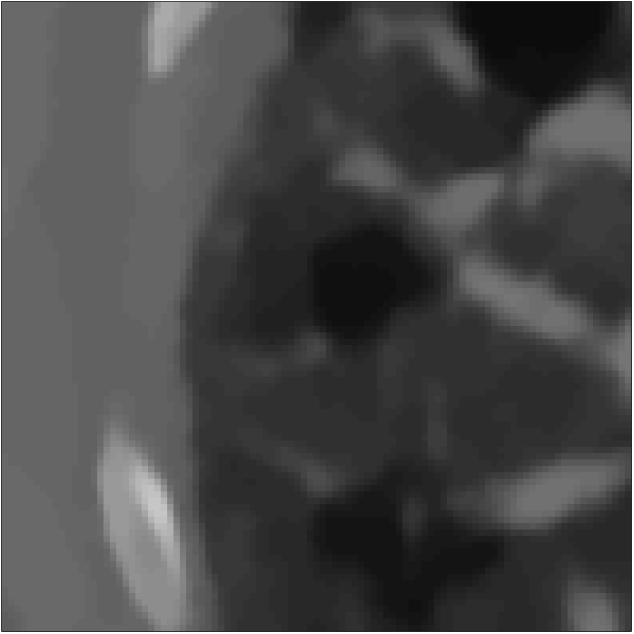}
    }
    \subfloat
    {
      \includegraphics[width=0.3\linewidth,
        angle=0]{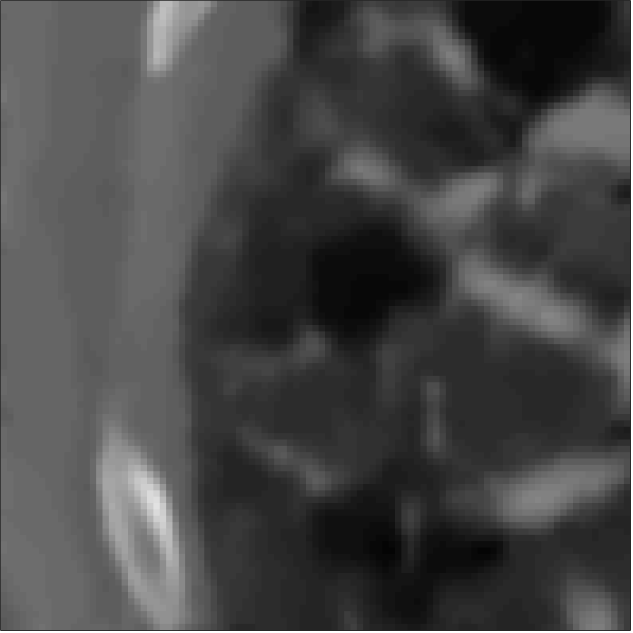}
    }
    \vspace{-0.02\linewidth}
  \end{minipage}
  \begin{minipage}{1.0\textwidth}
    \centering
    \subfloat
    {
      \includegraphics[width=0.3\linewidth,
        angle=0]{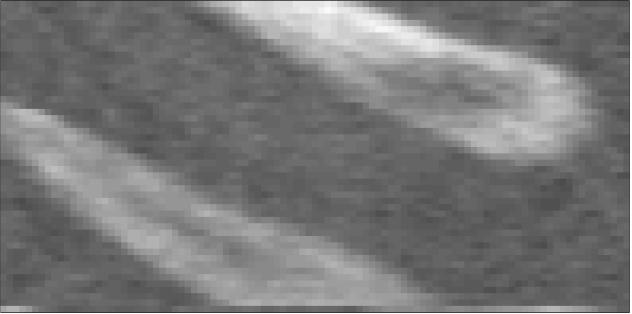}
    }
    \subfloat
    {
      \includegraphics[width=0.3\linewidth,
        angle=0]{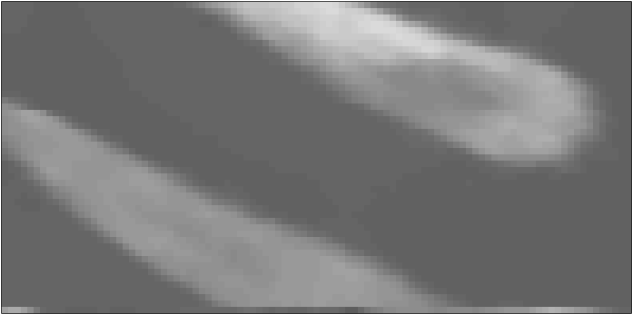}
    }
    \subfloat
    {
      \includegraphics[width=0.3\linewidth,
        angle=0]{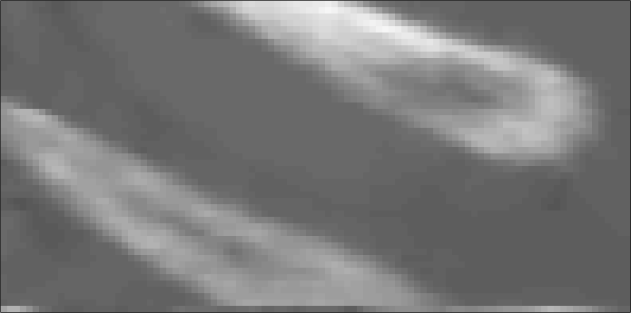}
    }
  \end{minipage}
  \caption{Visualization of computer tomography data of human lung tissue B via 3 groups of orthogonal slices. In contrast to \cref{fig:slicesb1}, different observation slices are presented. Row one, row two and row three present respectively the $xz$-slices, the $xy$-slices and the $yz$-slices. Column (a) visualizes the raw data with some noise (Gaussian noise). Column (b) visualizes the data denoised with ROF model. Column (c) visualizes the data denoised with proposed multidimensional TV-Stokes model.}
  \label{fig:slicesb2}
\end{figure}

\begin{figure}[!htbp]
  \captionsetup[subfigure]{justification=centering}
  \centering
  \begin{minipage}{1.0\textwidth}
    \centering
    \subfloat
    {
      \includegraphics[width=0.3\linewidth,
        height=0.3\linewidth,
        angle=0]{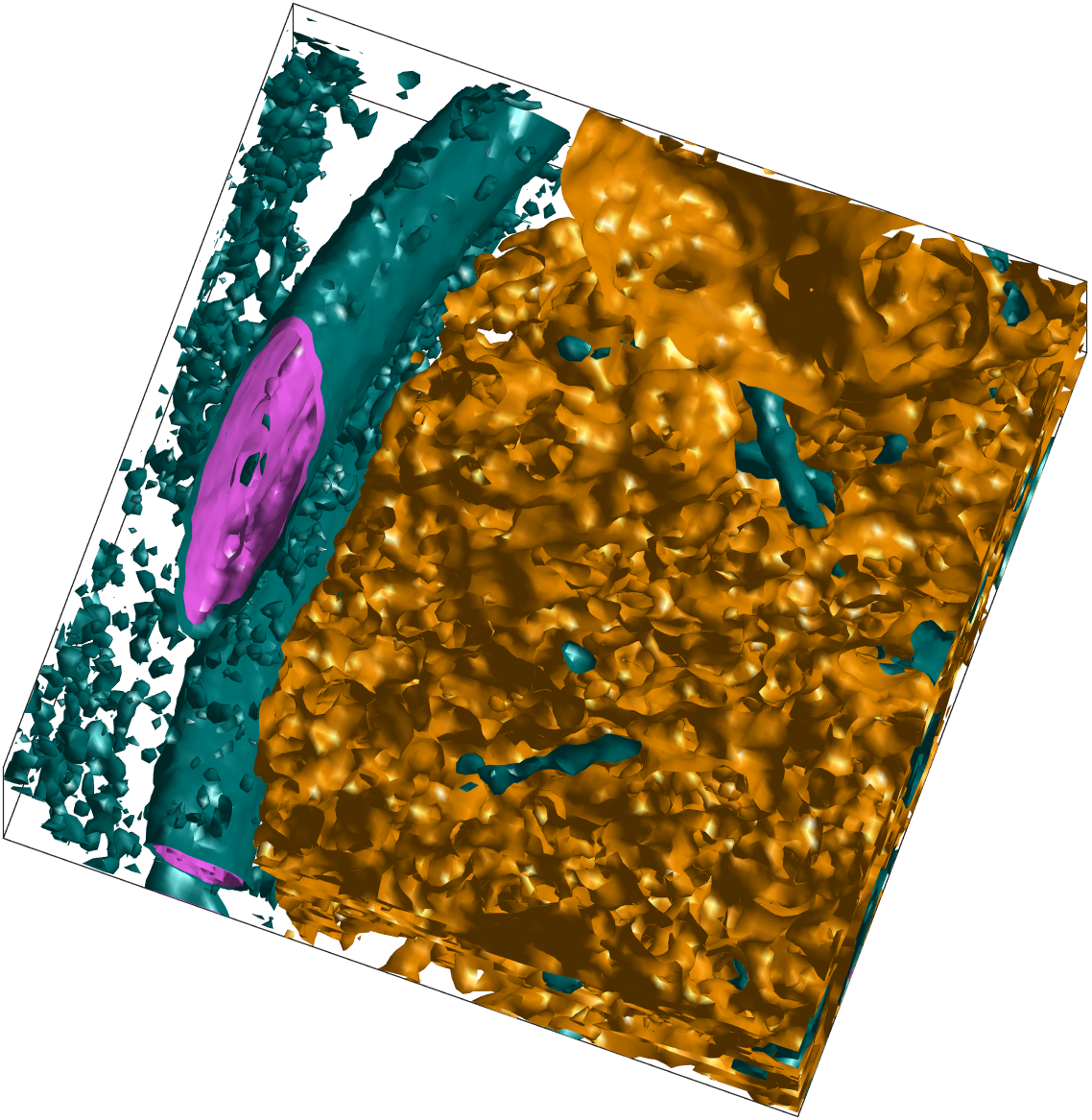}
    }
    \subfloat
    {
      \includegraphics[width=0.3\linewidth,
        height=0.3\linewidth,
        angle=0]{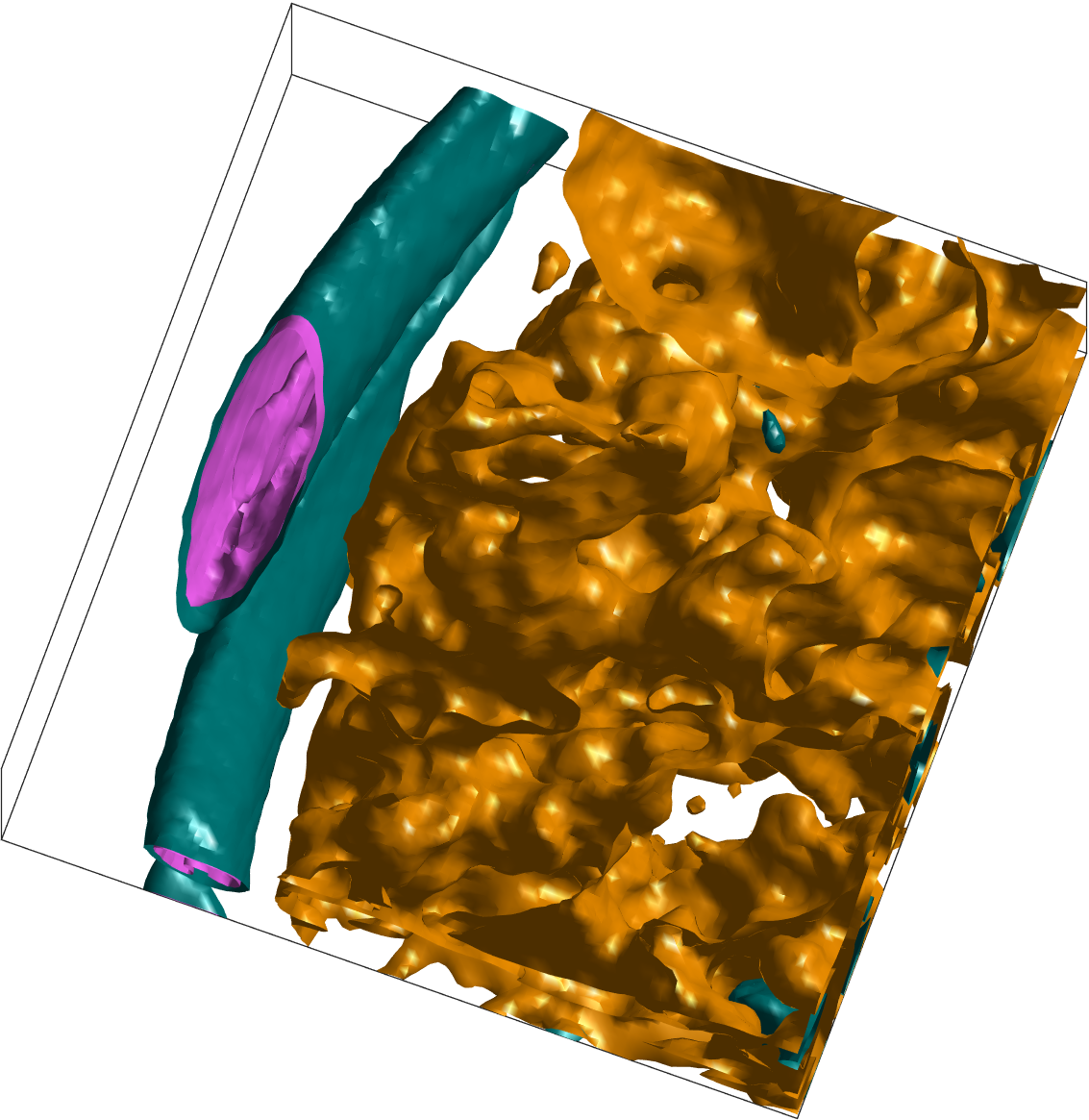}
    }
    \subfloat
    {
      \includegraphics[width=0.3\linewidth,
        height=0.3\linewidth,
        angle=0]{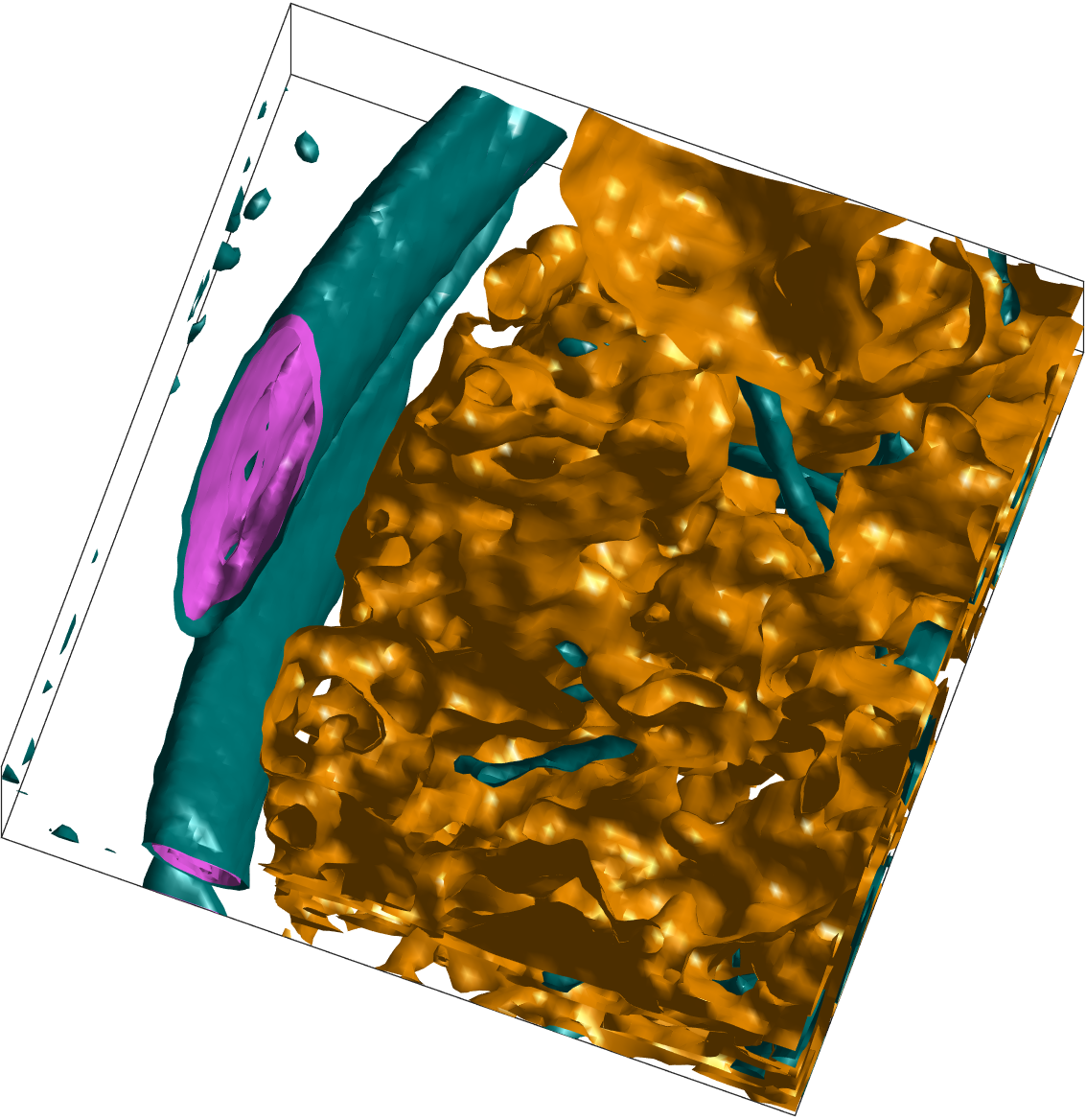}
    }
    \vspace{-0.02\linewidth}
  \end{minipage}
  \begin{minipage}{1.0\textwidth}
    \centering
    \subfloat
    {
      \includegraphics[width=0.3\linewidth,
        height=0.3\linewidth,
        angle=0]{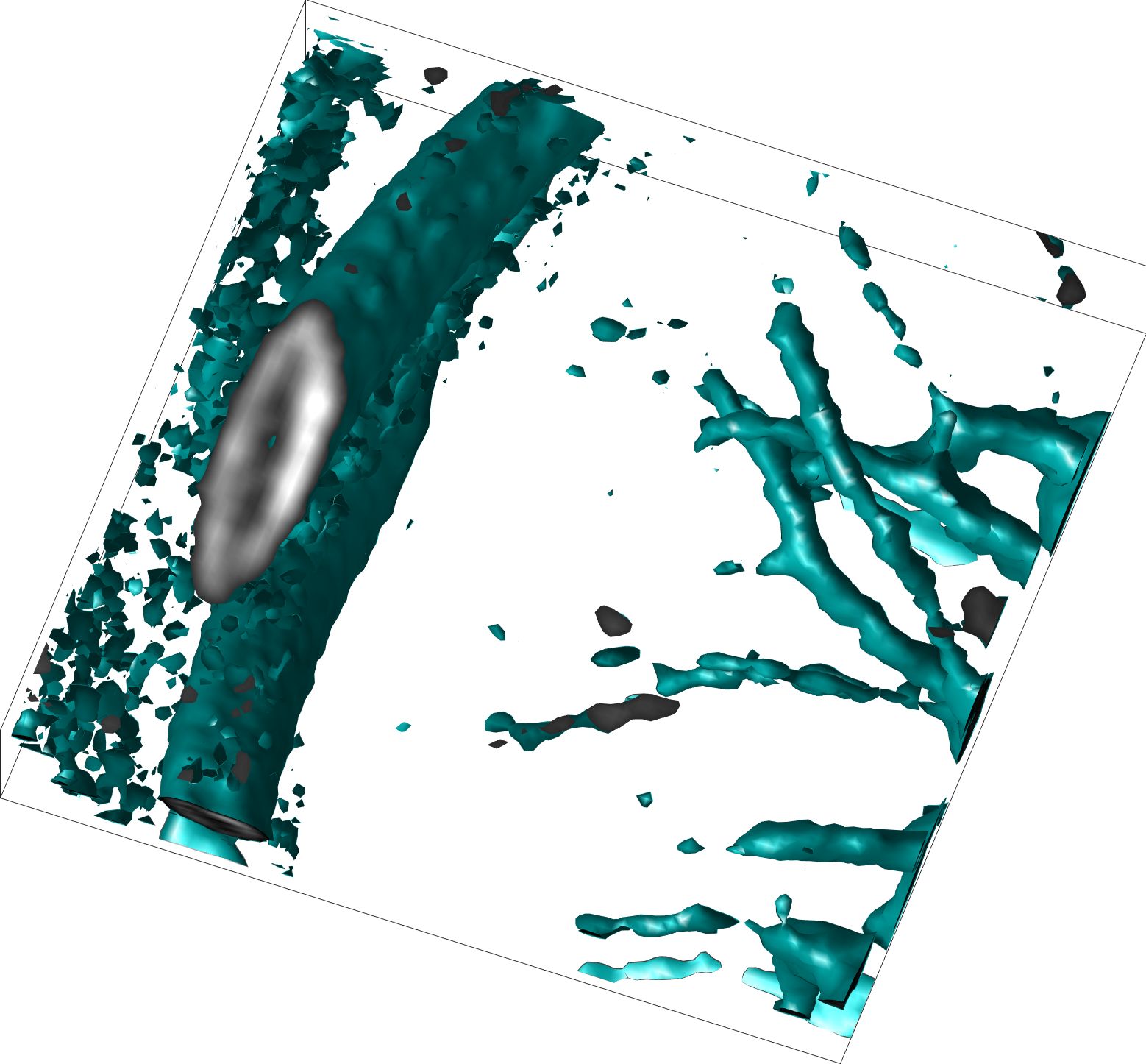}
    }
    \subfloat
    {
      \includegraphics[width=0.3\linewidth,
        height=0.3\linewidth,
        angle=0]{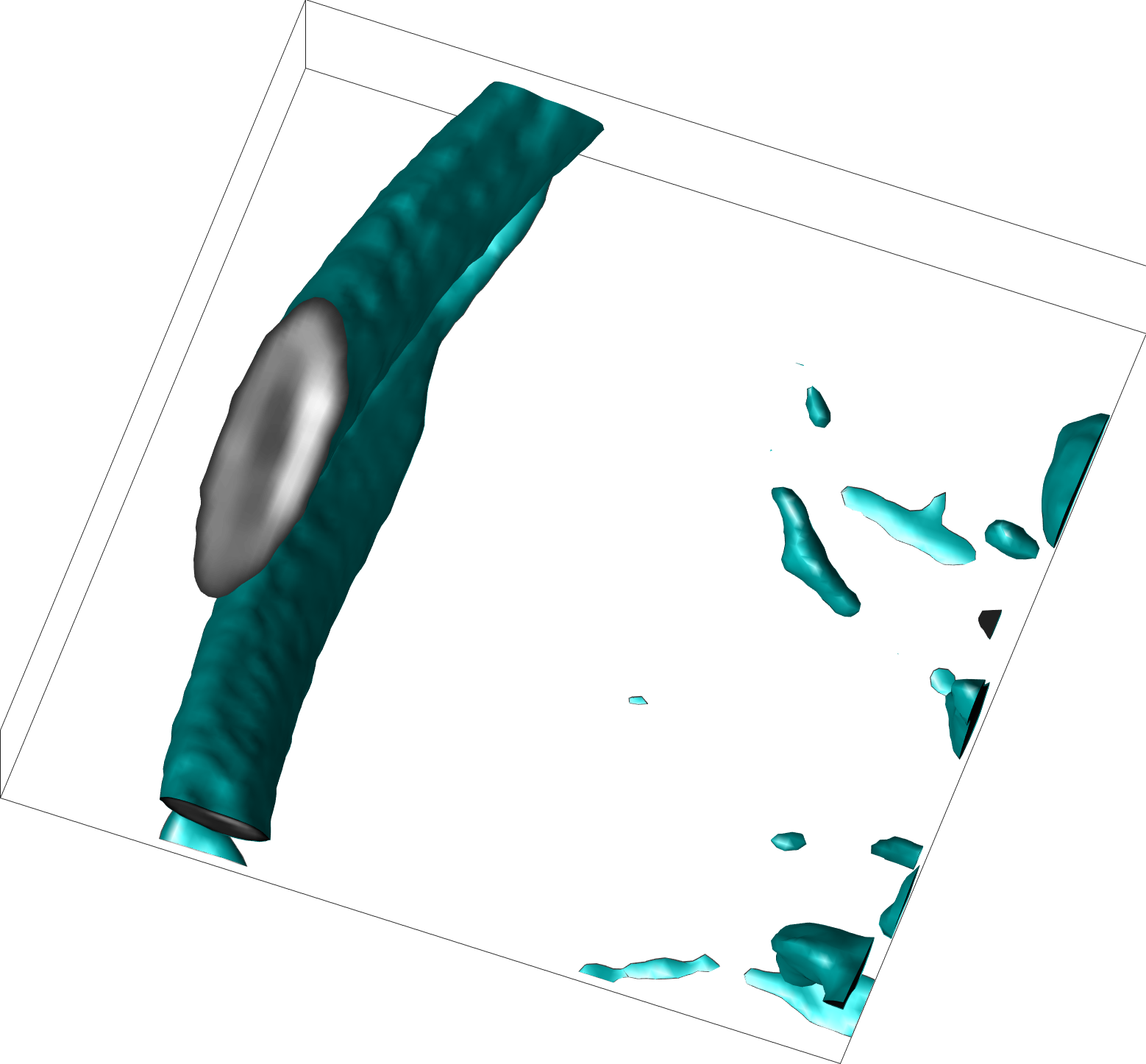}
    }
    \subfloat
    {
      \includegraphics[width=0.3\linewidth,
        height=0.3\linewidth,
        angle=0]{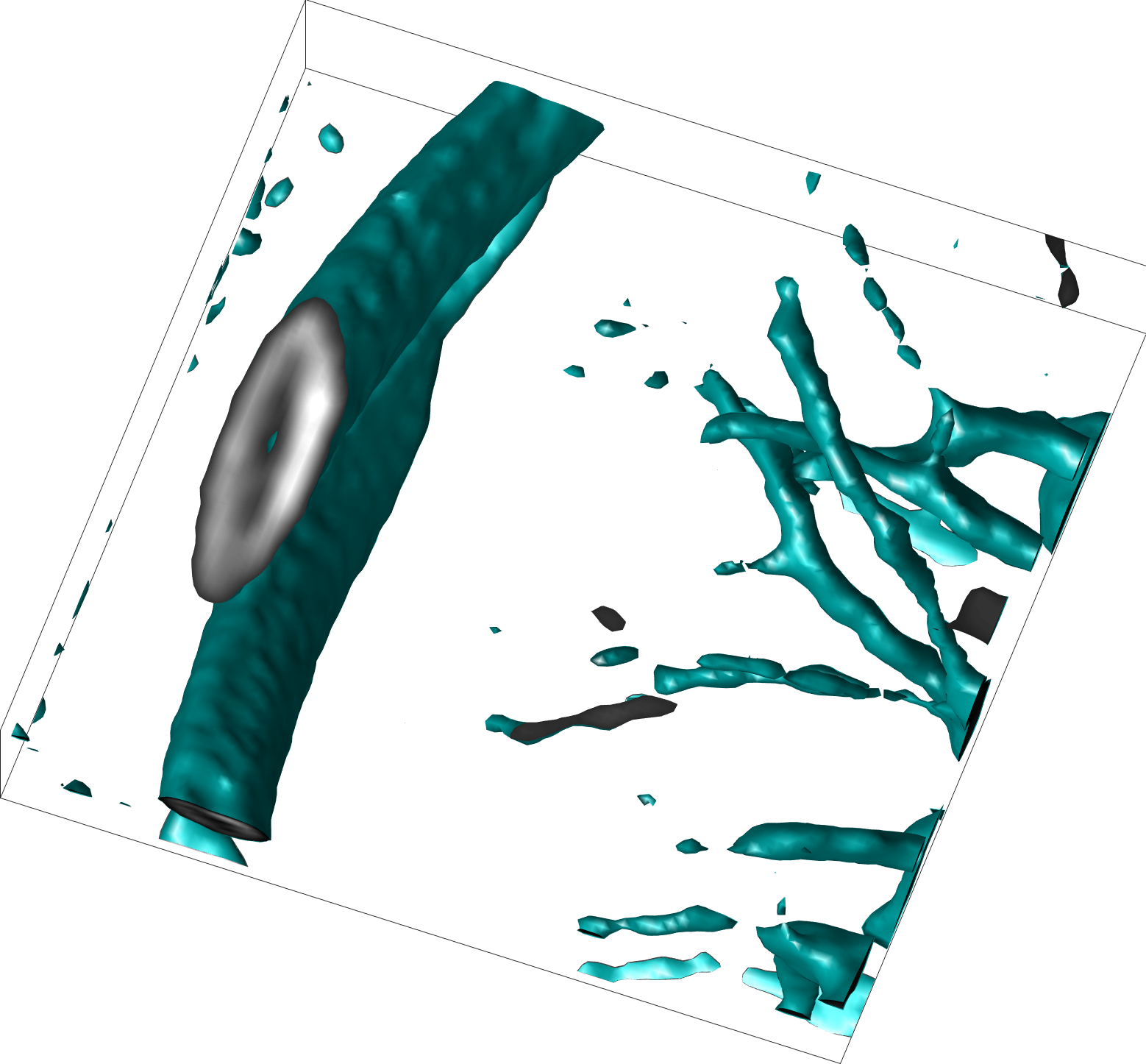}
    }
  \end{minipage}
  \caption{Isosurface show of computer tomography data of human lung tissue B. (a) raw data with some noise (Gaussian noise). (b) data denoised with ROF model. (c) data denoised with proposed multidimensional TV-Stokes model. In row two, only tube-like structures are presented.}
  \label{fig:isosurfaceb}
\end{figure}

Another experiment is performed on 2d dynamic data, where new 3d data are obtained by stacking all 2d data along temporal dimension. \cref{fig:hepburn1959} and \cref{fig:escalator} show the results of using proposed multidimensional TV-Stokes on movie / monitoring video denoising. In \cref{fig:hepburn1959}, smooth cheek and sharp edges are well balanced, showing in row two and row four. In escalator case, cf. \cref{fig:escalator}, shape edges of structures and digitals are preserved, showing in row two and row four.

\begin{figure}[!htbp]
  \captionsetup[subfigure]{justification=centering}
  \centering
  \begin{minipage}{1.0\textwidth}
    \centering
    \subfloat
    {
      \includegraphics[width=0.148\linewidth,
        height=0.12\linewidth,
        angle=0]{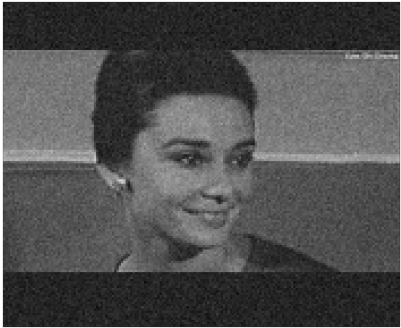}
    }
    \subfloat
    {
      \includegraphics[width=0.148\linewidth,
        height=0.12\linewidth,
        angle=0]{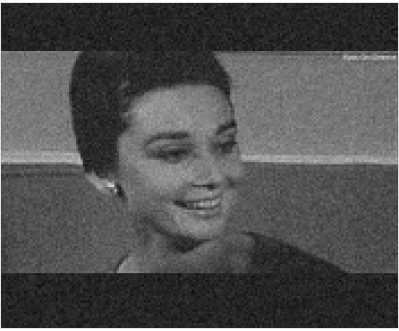}
    }
    \subfloat
    {
      \includegraphics[width=0.148\linewidth,
        height=0.12\linewidth,
        angle=0]{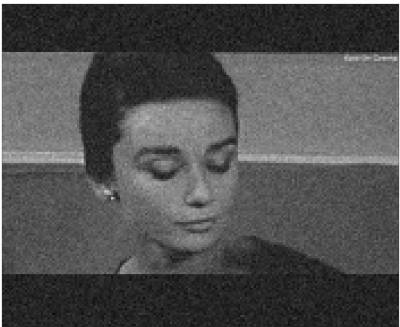}
    }
    \subfloat
    {
      \includegraphics[width=0.148\linewidth,
        height=0.12\linewidth,
        angle=0]{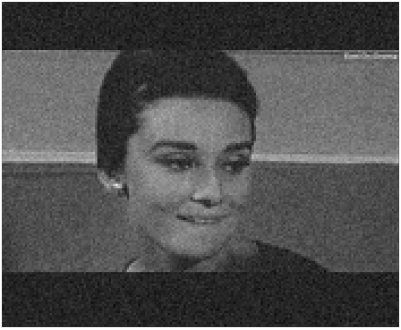}
    }
    \subfloat
    {
      \includegraphics[width=0.148\linewidth,
        height=0.12\linewidth,
        angle=0]{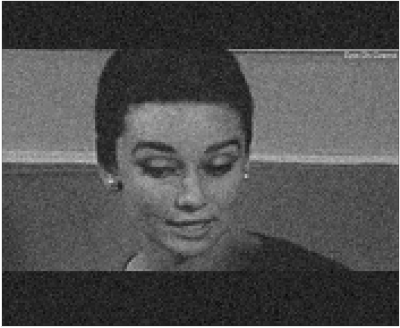}
    }
    \subfloat
    {
      \includegraphics[width=0.148\linewidth,
        height=0.12\linewidth,
        angle=0]{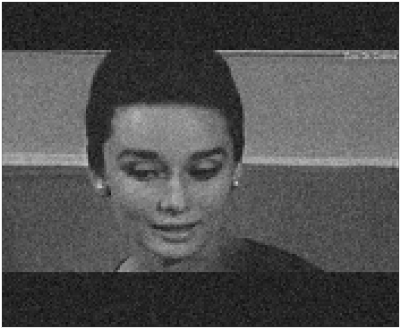}
    }
    \vspace{-0.02\linewidth}
  \end{minipage}
  \begin{minipage}{1.0\textwidth}
    \centering
    \subfloat
    {
      \includegraphics[width=0.148\linewidth,
        height=0.12\linewidth,
        angle=0]{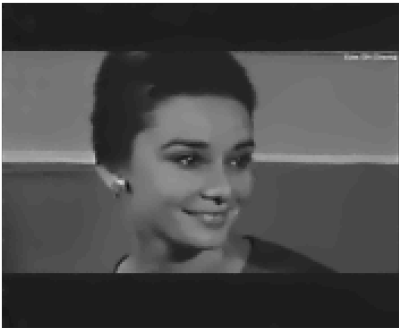}
    }
    \subfloat
    {
      \includegraphics[width=0.148\linewidth,
        height=0.12\linewidth,
        angle=0]{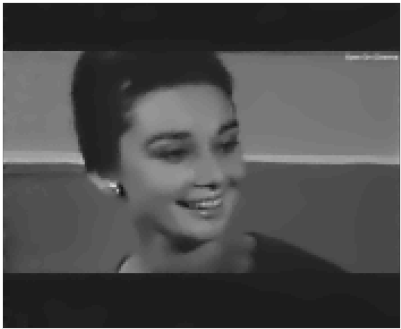}
    }
    \subfloat
    {
      \includegraphics[width=0.148\linewidth,
        height=0.12\linewidth,
        angle=0]{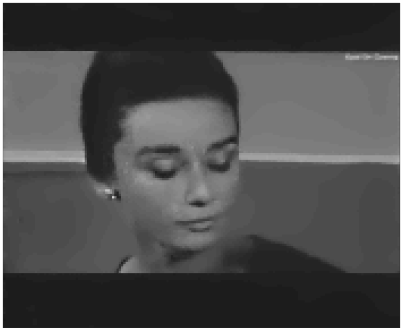}
    }
    \subfloat
    {
      \includegraphics[width=0.148\linewidth,
        height=0.12\linewidth,
        angle=0]{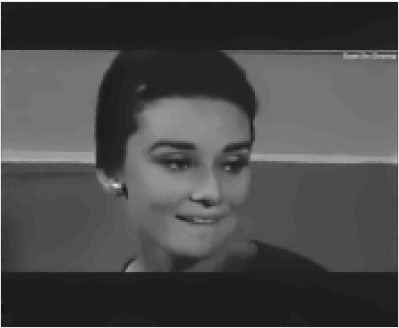}
    }
    \subfloat
    {
      \includegraphics[width=0.148\linewidth,
        height=0.12\linewidth,
        angle=0]{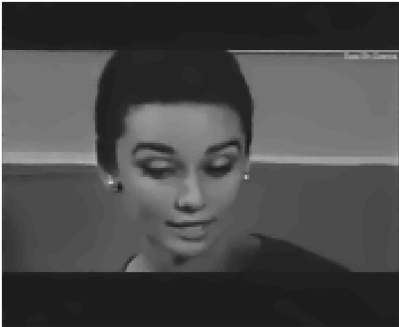}
    }
    \subfloat
    {
      \includegraphics[width=0.148\linewidth,
        height=0.12\linewidth,
        angle=0]{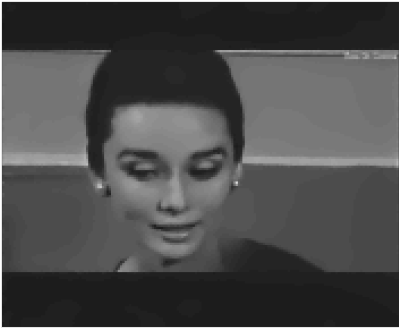}
    }
    \vspace{-0.02\linewidth}
  \end{minipage}
  \begin{minipage}{1.0\textwidth}
    \centering
    \subfloat
    {
      \includegraphics[width=0.148\linewidth,
        height=0.12\linewidth,
        angle=0]{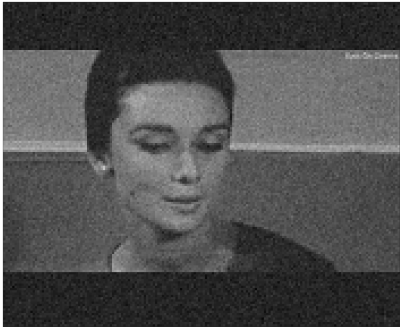}
    }
    \subfloat
    {
      \includegraphics[width=0.148\linewidth,
        height=0.12\linewidth,
        angle=0]{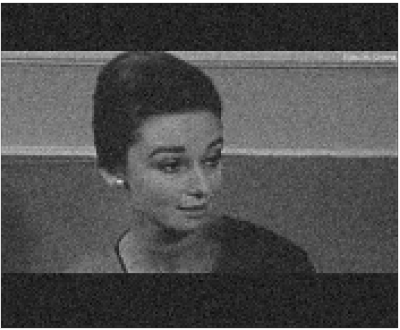}
    }
    \subfloat
    {
      \includegraphics[width=0.148\linewidth,
        height=0.12\linewidth,
        angle=0]{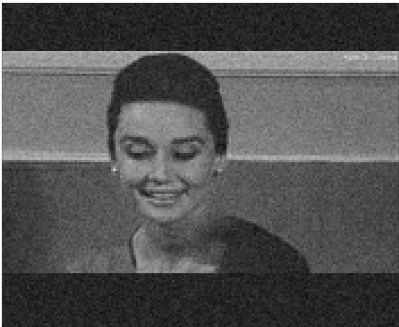}
    }
    \subfloat
    {
      \includegraphics[width=0.148\linewidth,
        height=0.12\linewidth,
        angle=0]{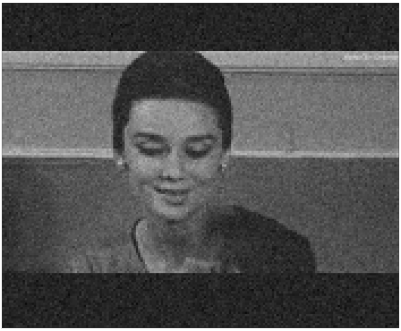}
    }
    \subfloat
    {
      \includegraphics[width=0.148\linewidth,
        height=0.12\linewidth,
        angle=0]{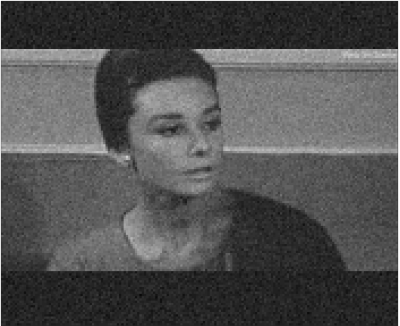}
    }
    \subfloat
    {
      \includegraphics[width=0.148\linewidth,
        height=0.12\linewidth,
        angle=0]{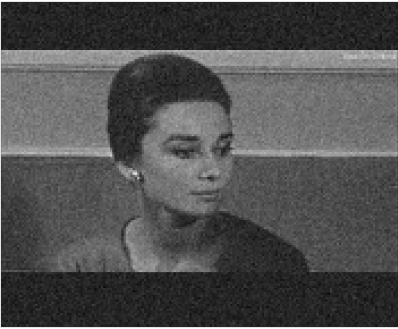}
    }
    \vspace{-0.02\linewidth}
  \end{minipage}
  \begin{minipage}{1.0\textwidth}
    \centering
    \subfloat
    {
      \includegraphics[width=0.148\linewidth,
        height=0.12\linewidth,
        angle=0]{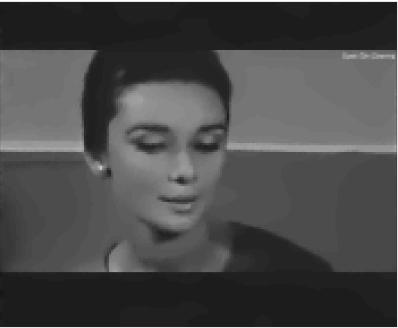}
    }
    \subfloat
    {
      \includegraphics[width=0.148\linewidth,
        height=0.12\linewidth,
        angle=0]{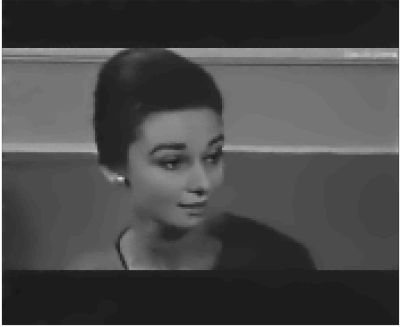}
    }
    \subfloat
    {
      \includegraphics[width=0.148\linewidth,
        height=0.12\linewidth,
        angle=0]{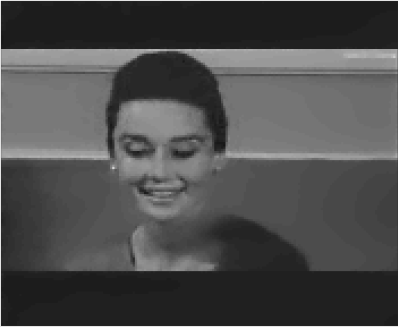}
    }
    \subfloat
    {
      \includegraphics[width=0.148\linewidth,
        height=0.12\linewidth,
        angle=0]{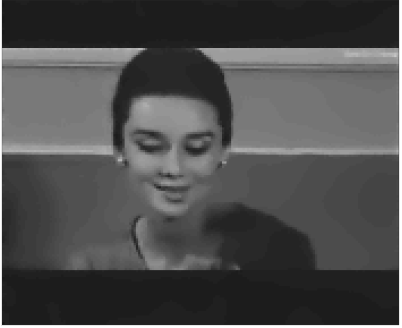}
    }
    \subfloat
    {
      \includegraphics[width=0.148\linewidth,
        height=0.12\linewidth,
        angle=0]{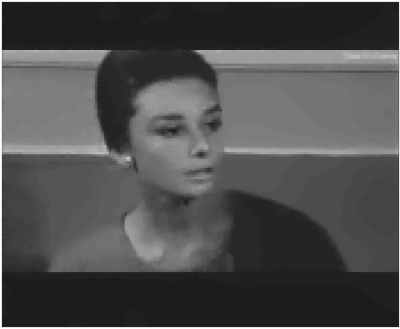}
    }
    \subfloat
    {
      \includegraphics[width=0.148\linewidth,
        height=0.12\linewidth,
        angle=0]{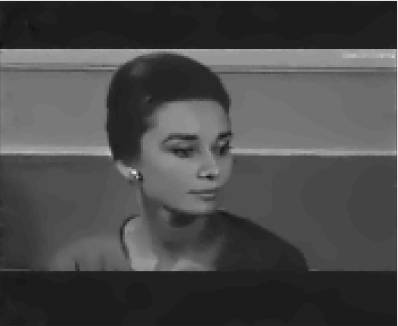}
    }
  \end{minipage}
  \caption{Contrast of initial noisy movie and its denoised copy, by using proposed multidimensional TV-Stokes model. Row one and row three are sequences from initial movie. Row two and row four are sequences from denoised results.}
  \label{fig:hepburn1959}
\end{figure}

\begin{figure}[!htbp]
  \captionsetup[subfigure]{justification=centering}
  \centering
  \begin{minipage}{1.0\textwidth}
    \centering
    \subfloat
    {
      \includegraphics[width=0.18\linewidth,
        angle=0]{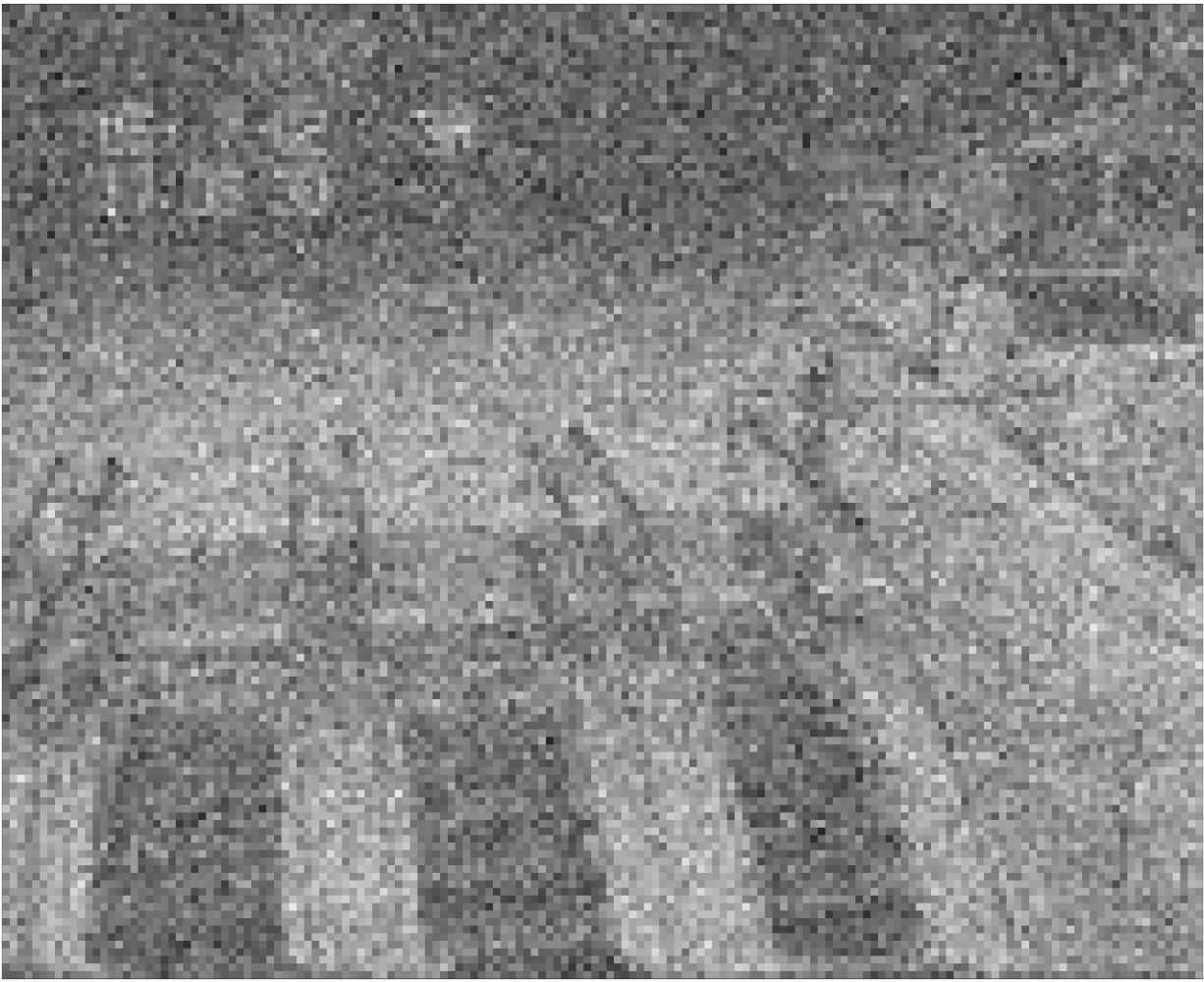}
    }
    \subfloat
    {
      \includegraphics[width=0.18\linewidth,
        angle=0]{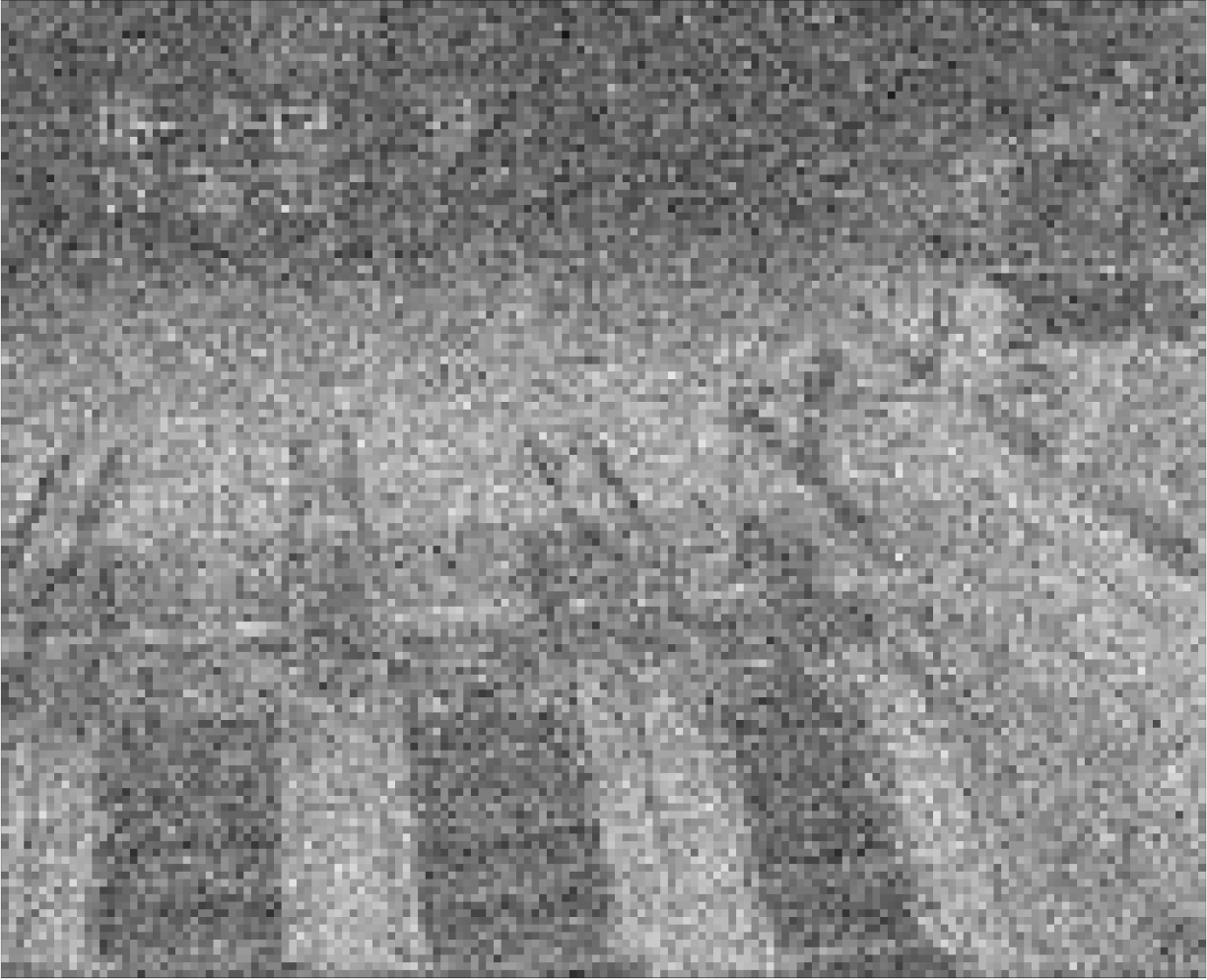}
    }
    \subfloat
    {
      \includegraphics[width=0.18\linewidth,
        angle=0]{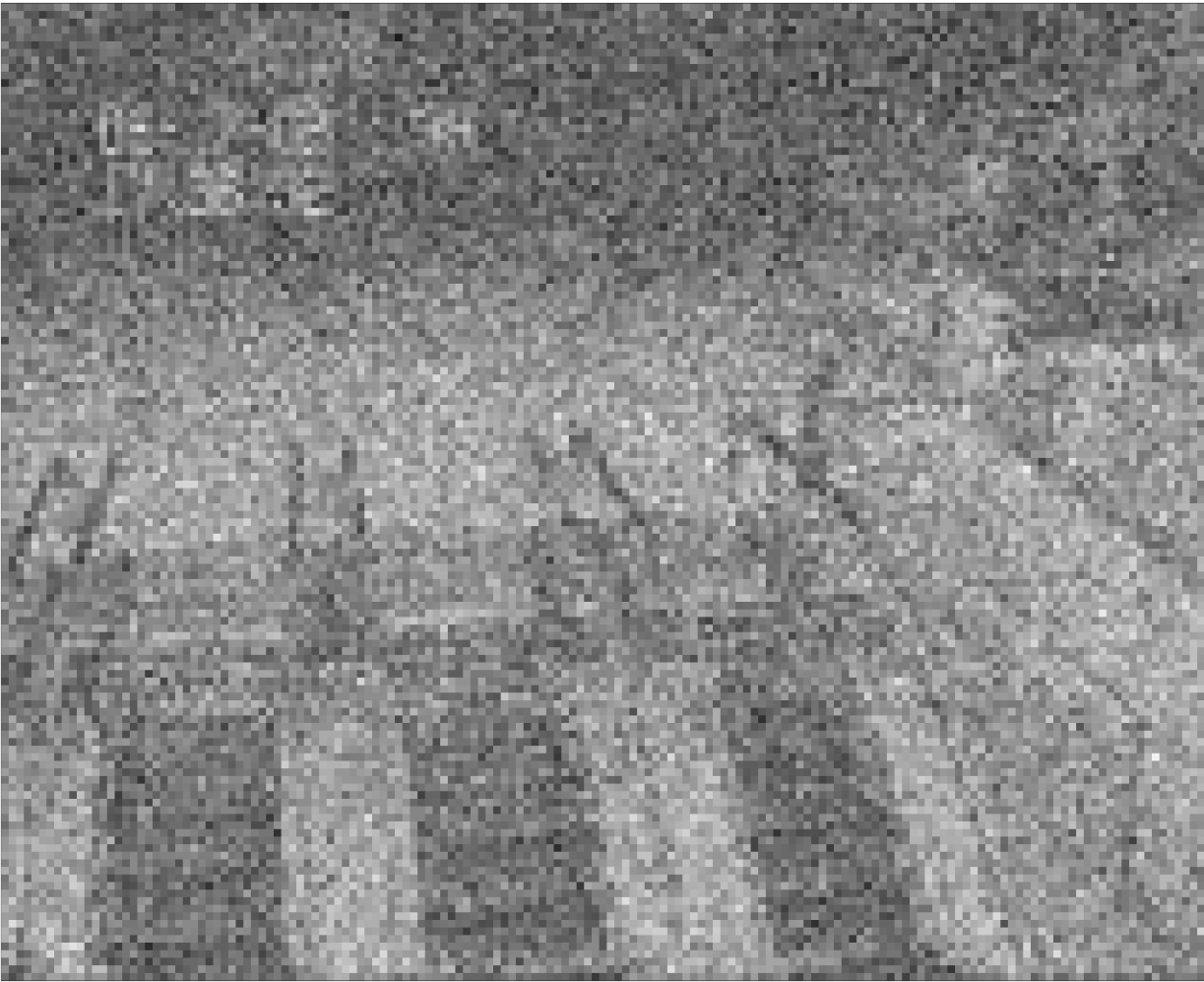}
    }
    \subfloat
    {
      \includegraphics[width=0.18\linewidth,
        angle=0]{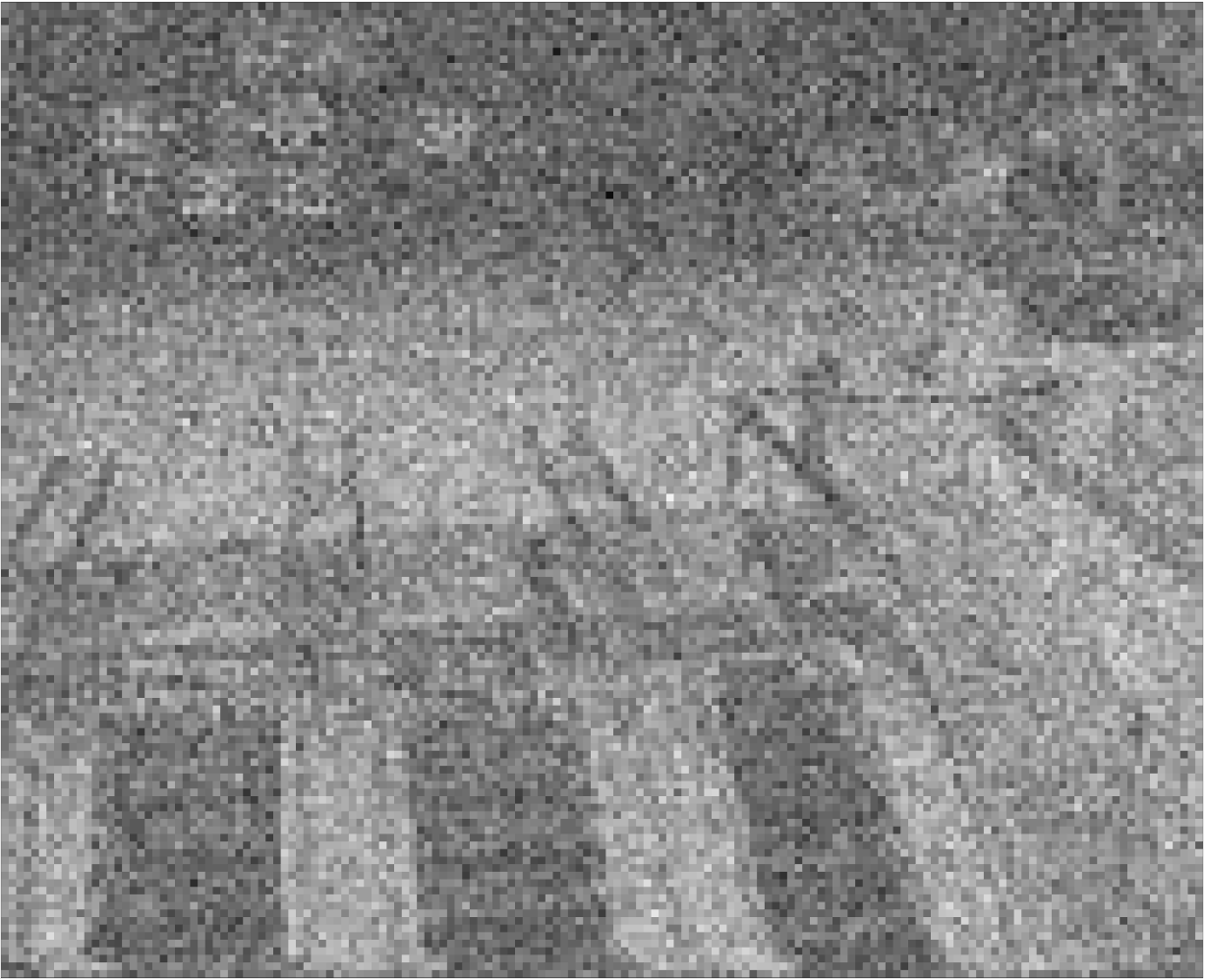}
    }
    \subfloat
    {
      \includegraphics[width=0.18\linewidth,
        angle=0]{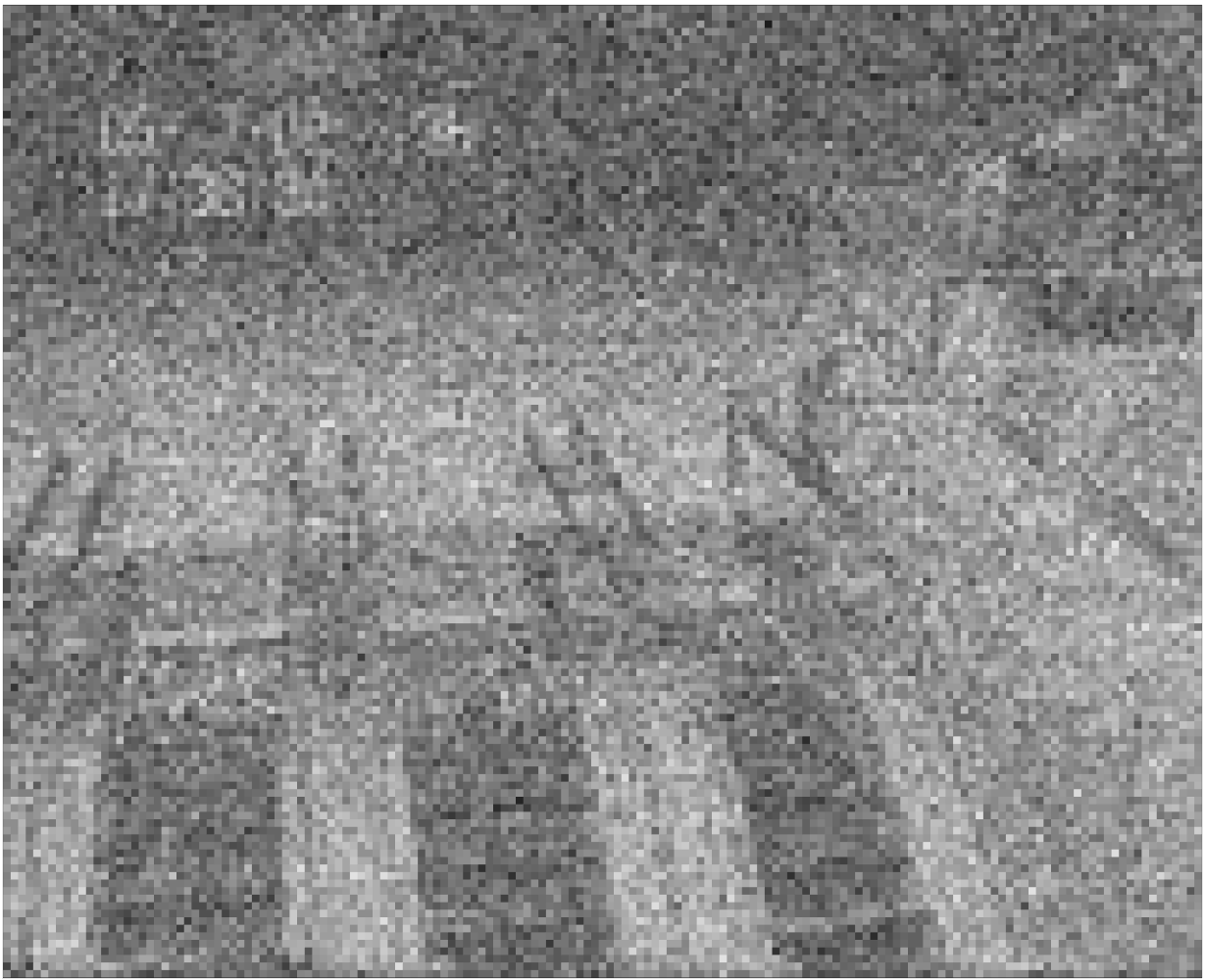}
    }
    \vspace{-0.02\linewidth}
  \end{minipage}
  \begin{minipage}{1.0\textwidth}
    \centering
    \subfloat
    {
      \includegraphics[width=0.18\linewidth,
        angle=0]{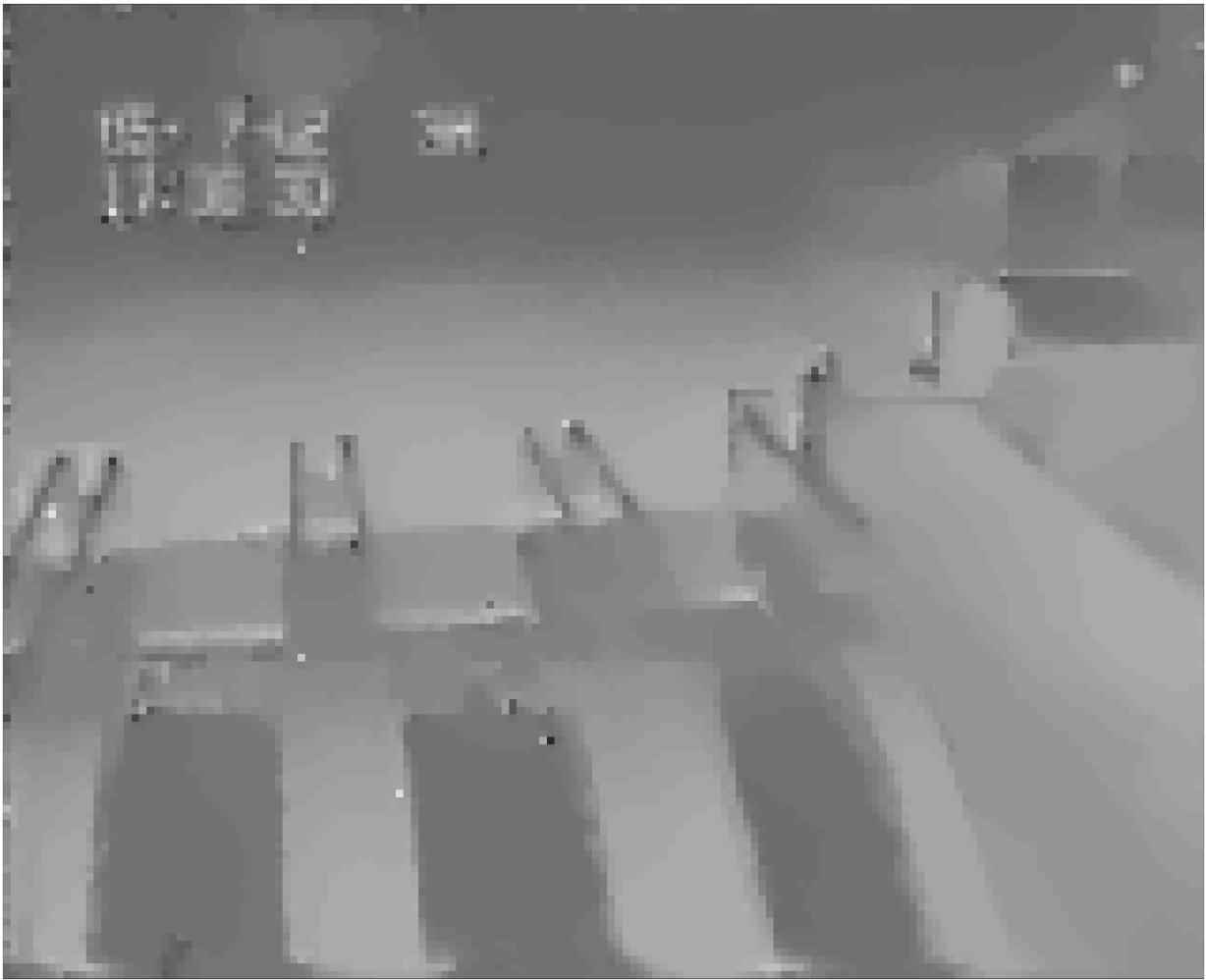}
    }
    \subfloat
    {
      \includegraphics[width=0.18\linewidth,
        angle=0]{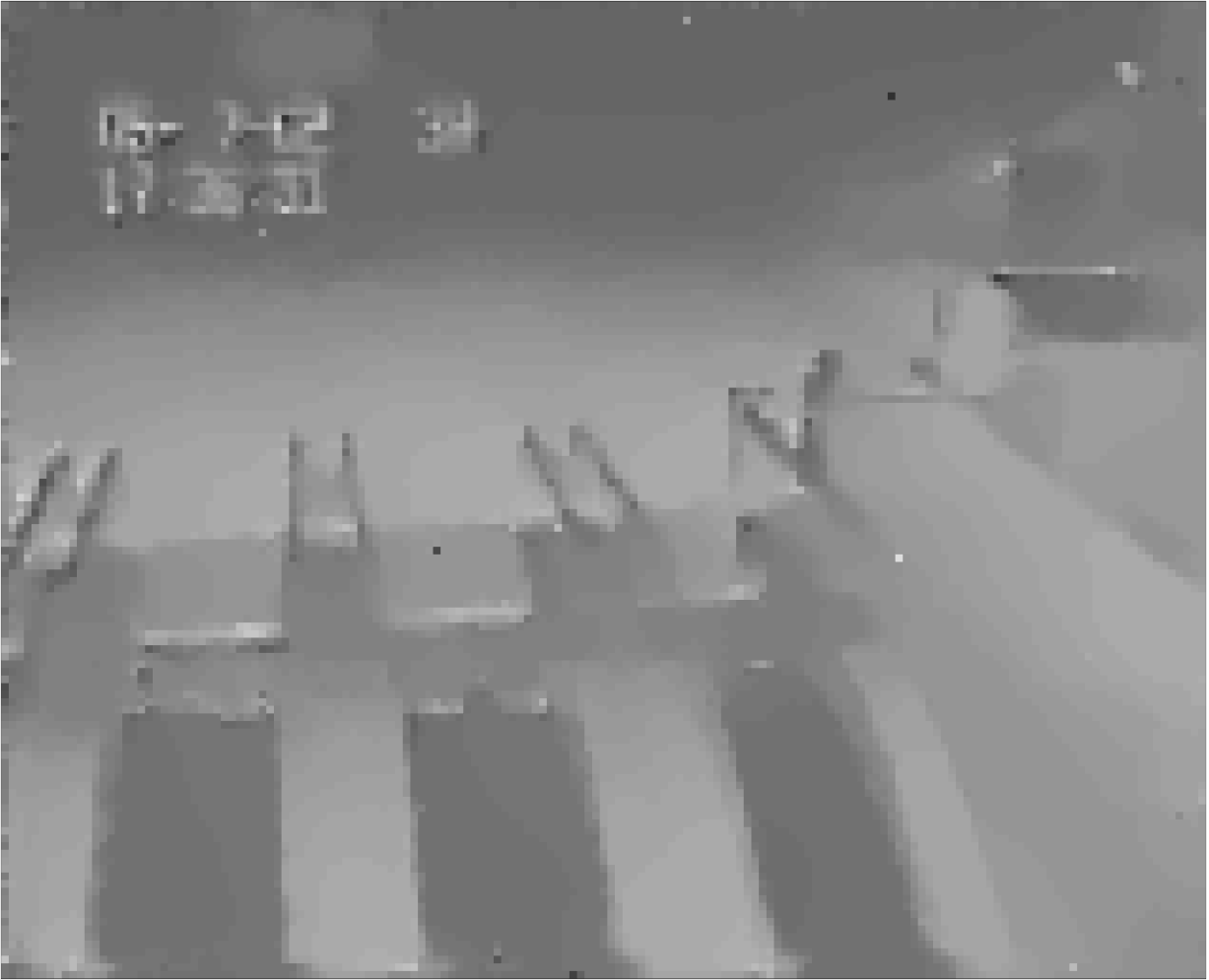}
    }
    \subfloat
    {
      \includegraphics[width=0.18\linewidth,
        angle=0]{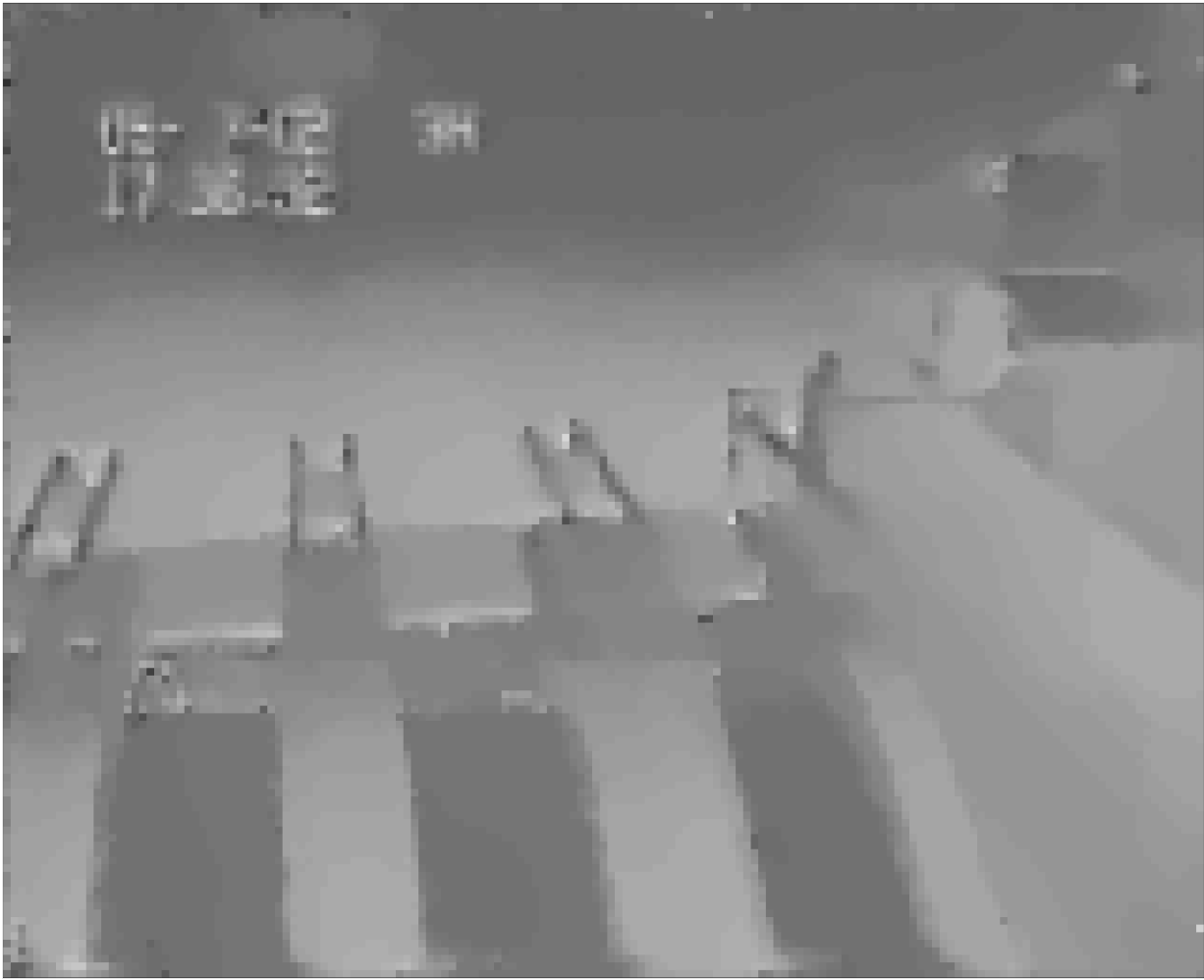}
    }
    \subfloat
    {
      \includegraphics[width=0.18\linewidth,
        angle=0]{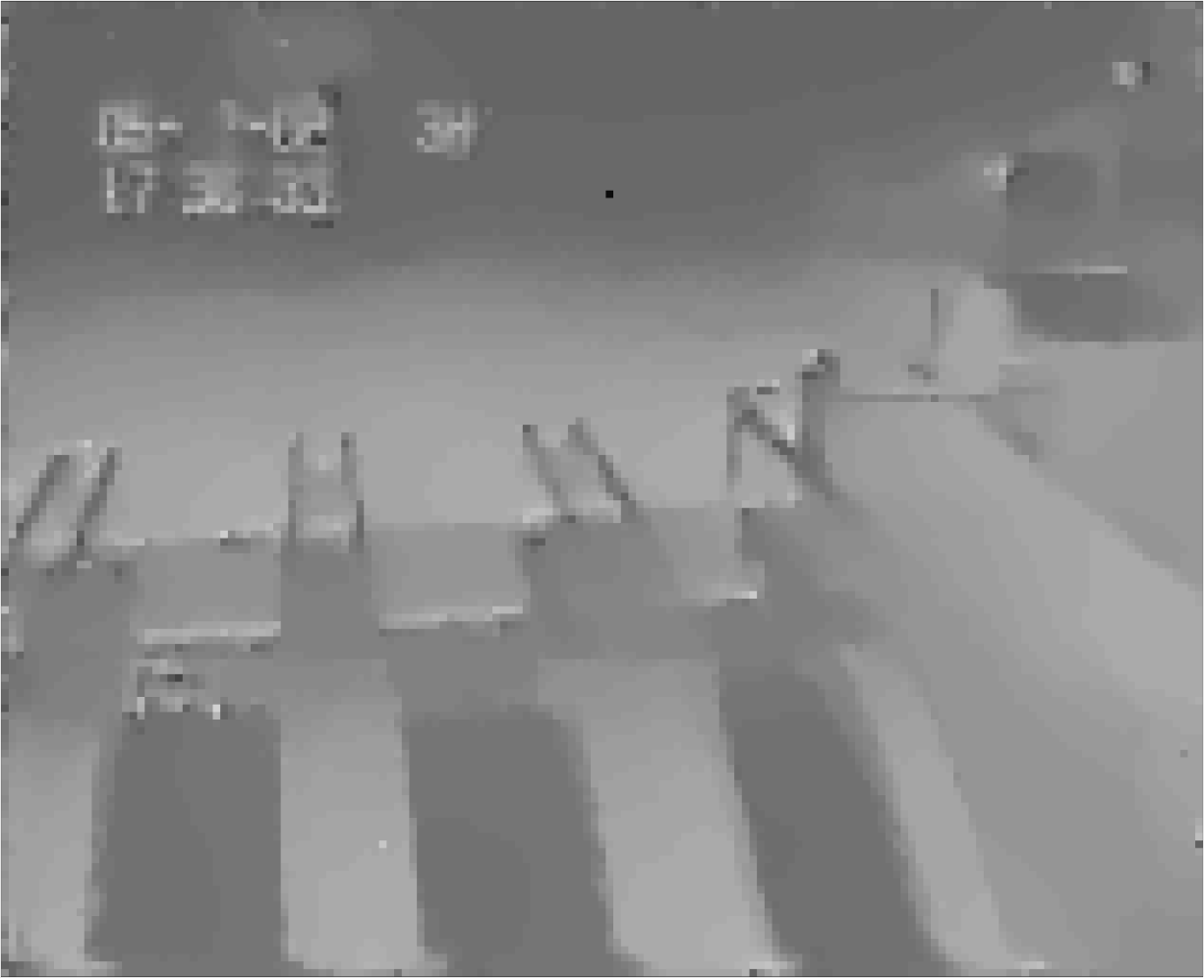}
    }
    \subfloat
    {
      \includegraphics[width=0.18\linewidth,
        angle=0]{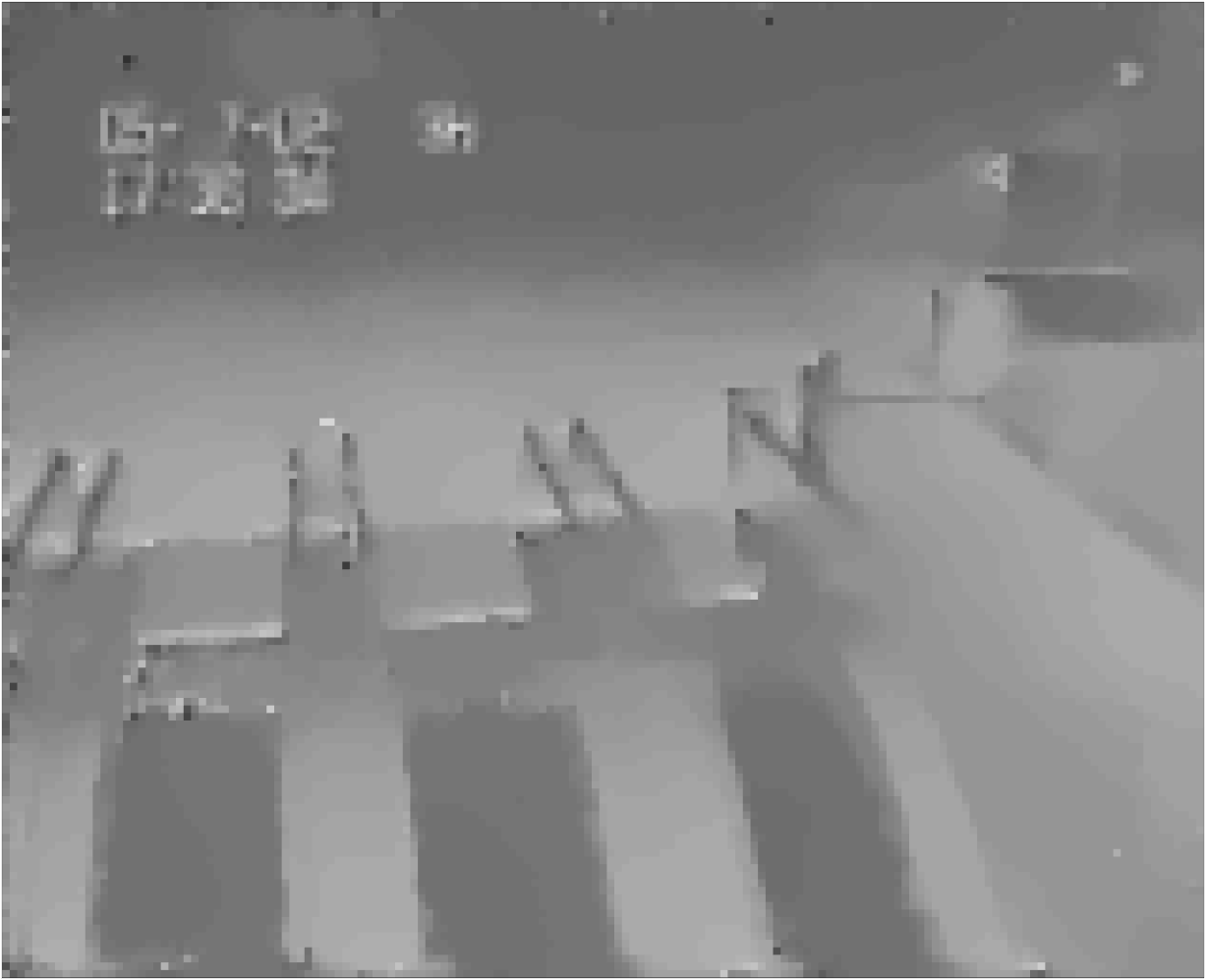}
    }
    \vspace{-0.02\linewidth}
  \end{minipage}
  \begin{minipage}{1.0\textwidth}
    \centering
    \subfloat
    {
      \includegraphics[width=0.18\linewidth,
        angle=0]{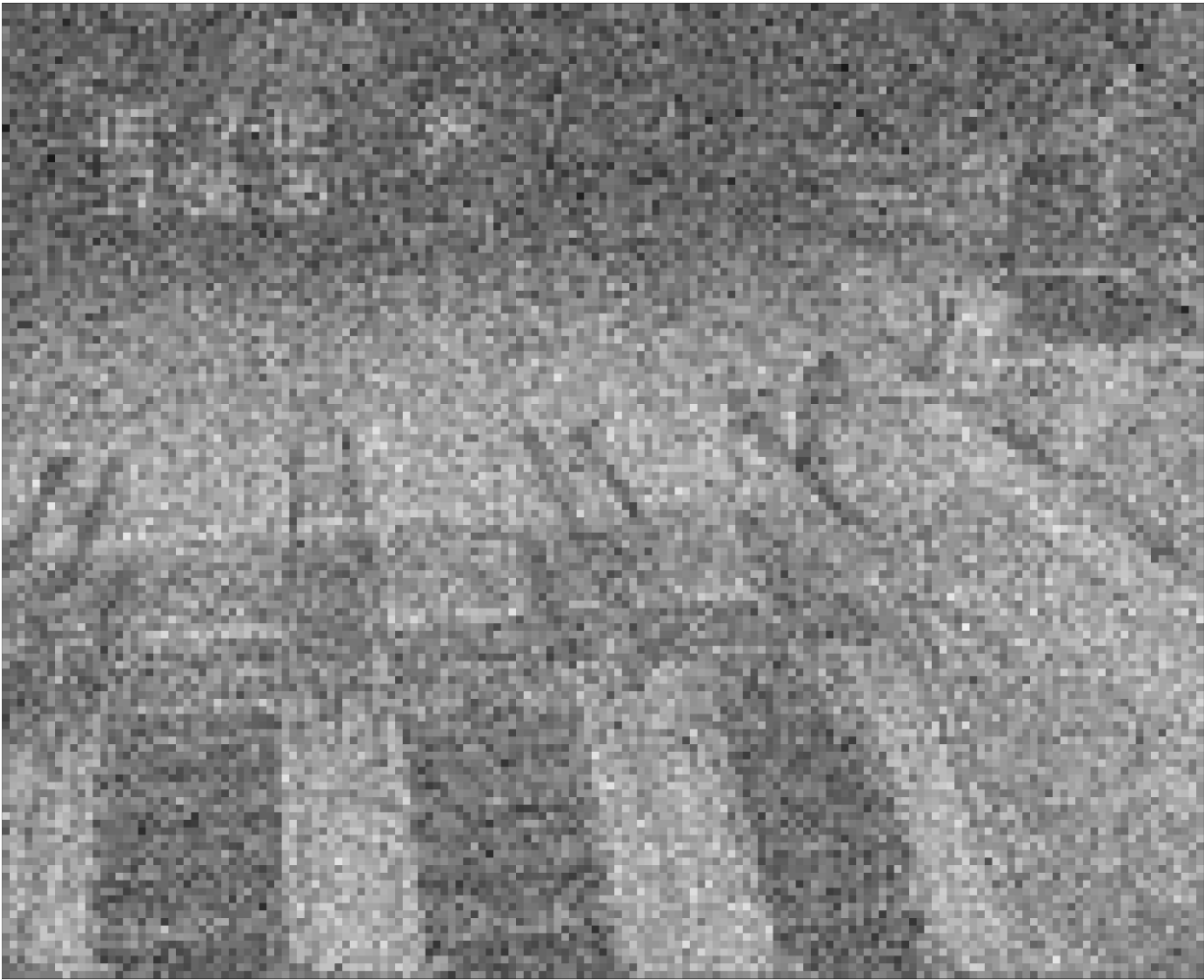}
    }
    \subfloat
    {
      \includegraphics[width=0.18\linewidth,
        angle=0]{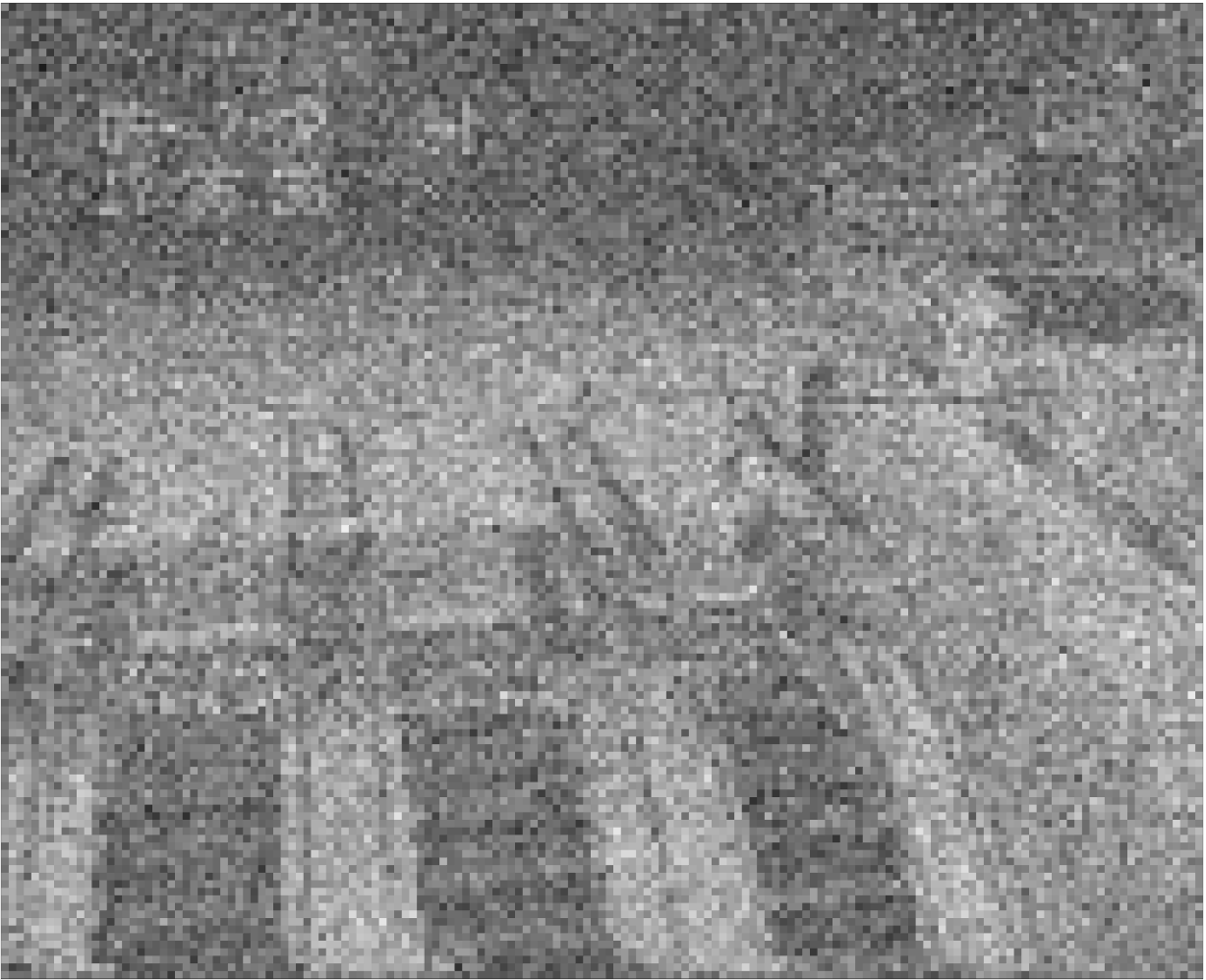}
    }
    \subfloat
    {
      \includegraphics[width=0.18\linewidth,
        angle=0]{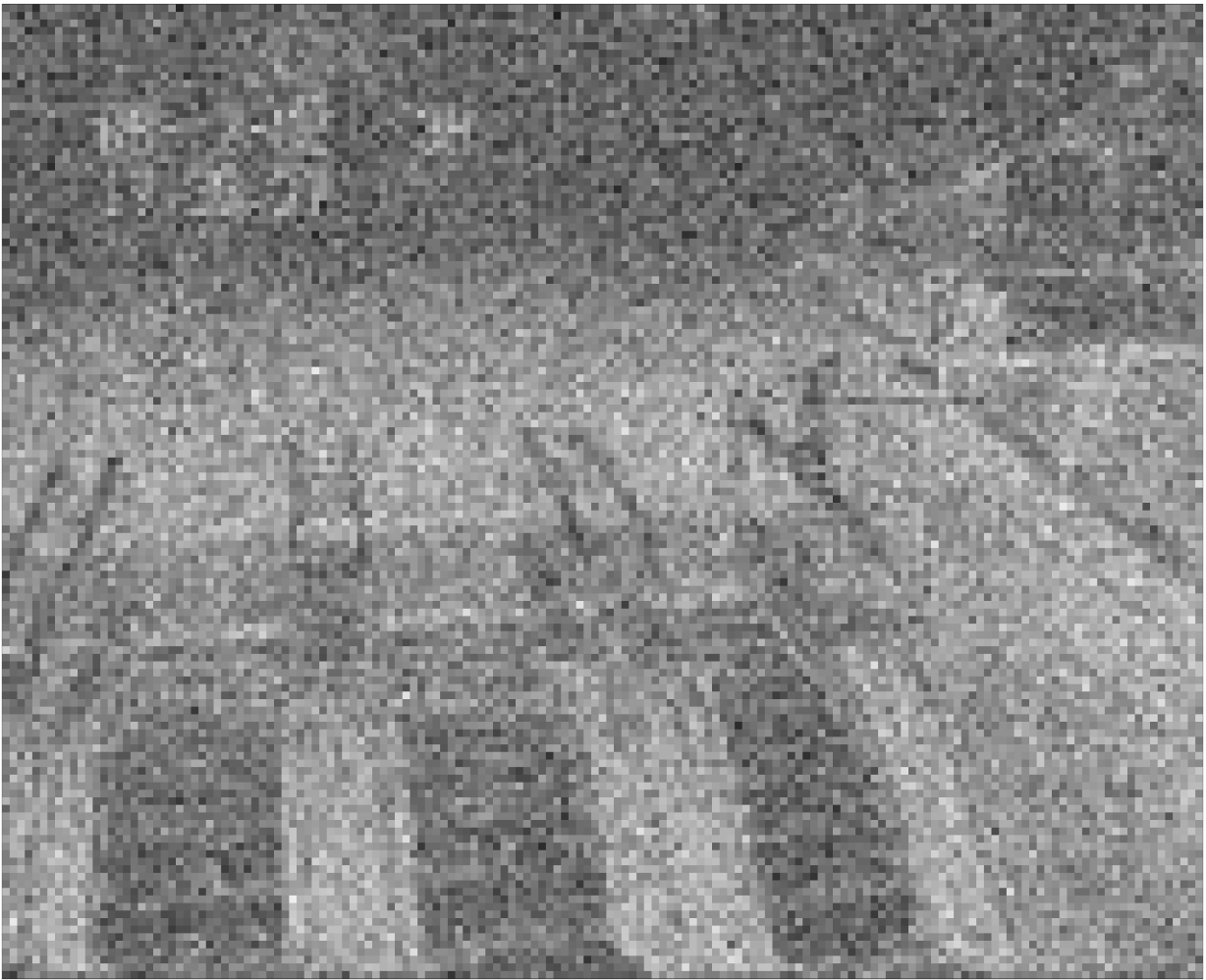}
    }
    \subfloat
    {
      \includegraphics[width=0.18\linewidth,
        angle=0]{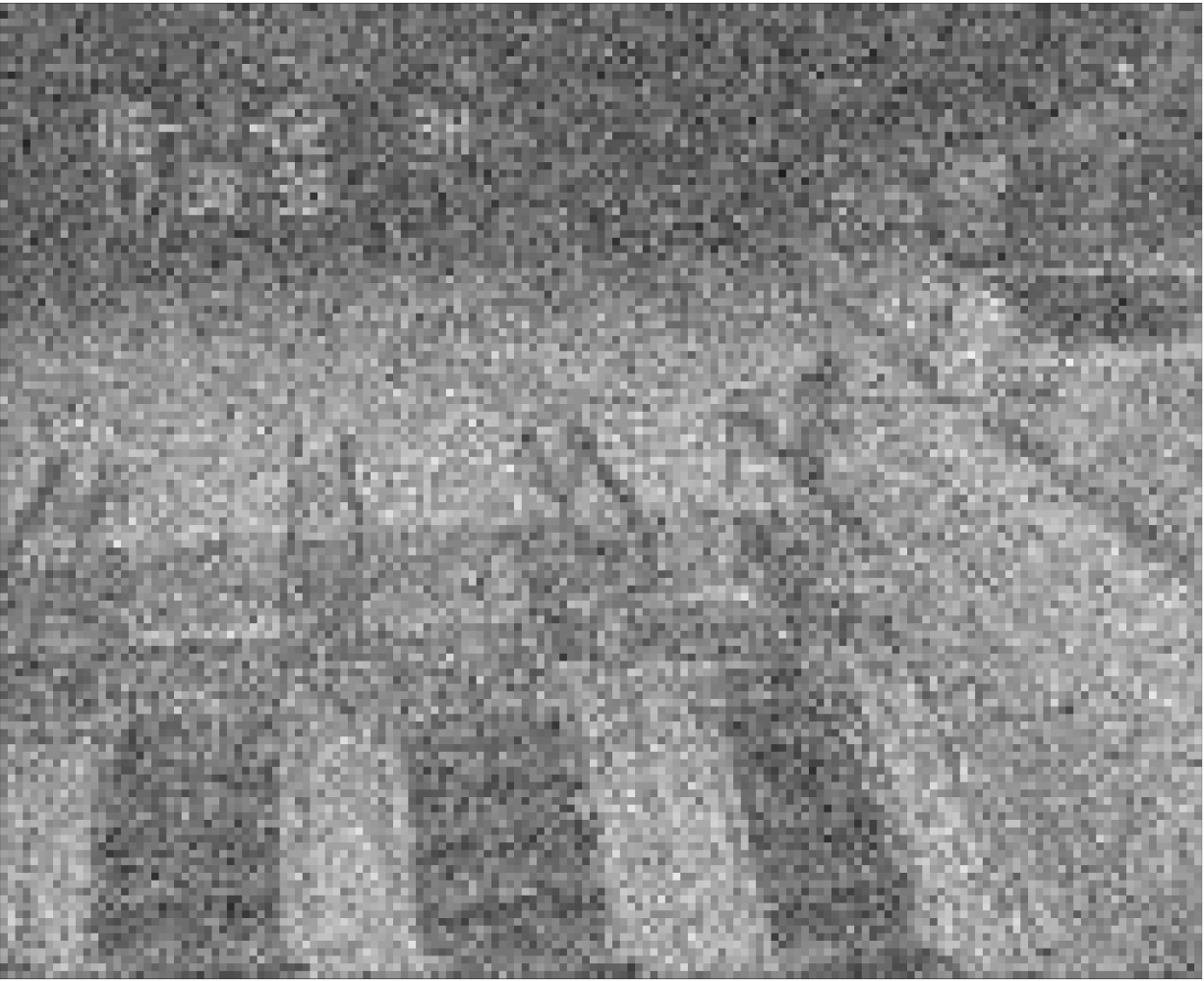}
    }
    \subfloat
    {
      \includegraphics[width=0.18\linewidth,
        angle=0]{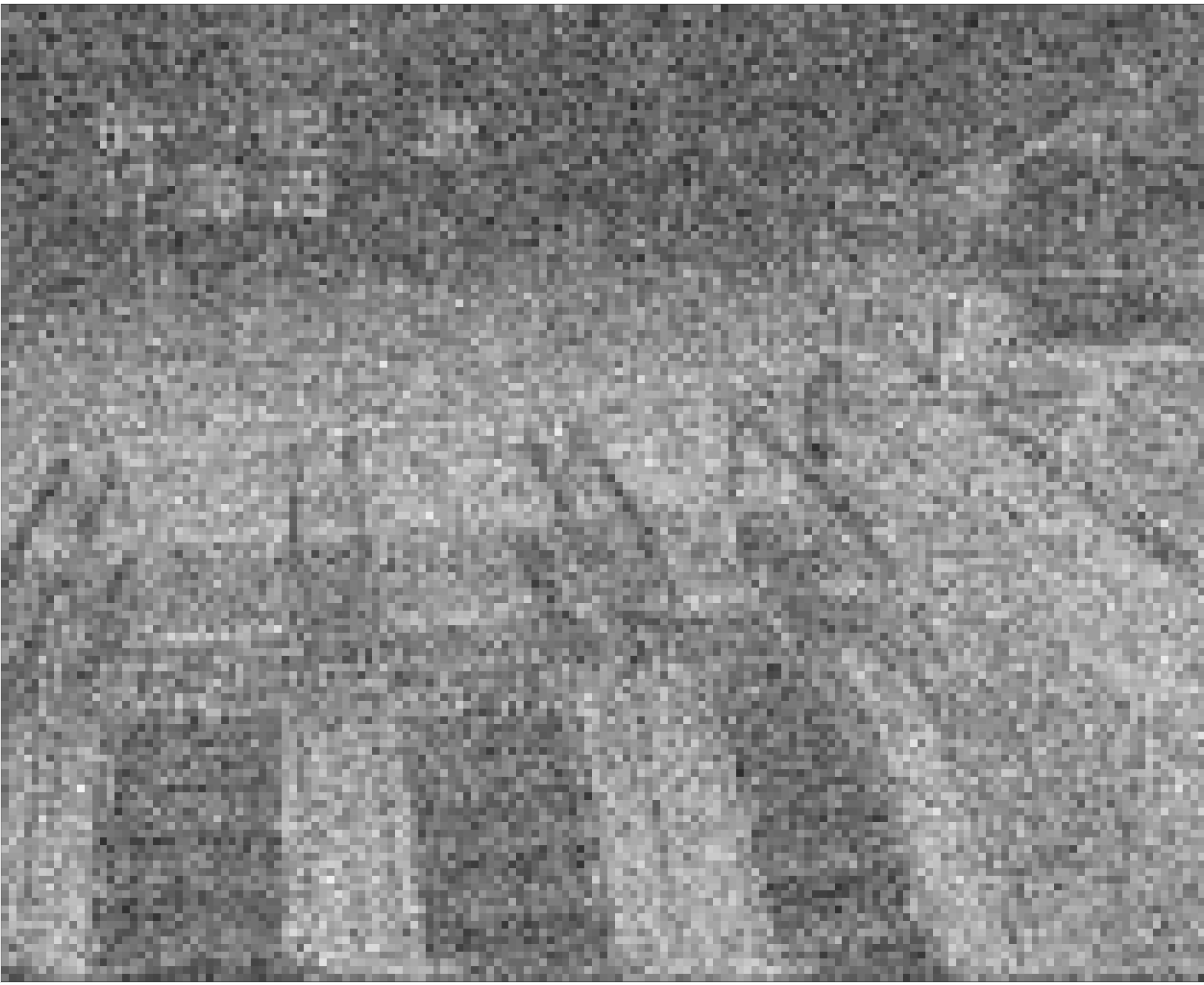}
    }
    \vspace{-0.02\linewidth}
  \end{minipage}
  \begin{minipage}{1.0\textwidth}
    \centering
    \subfloat
    {
      \includegraphics[width=0.18\linewidth,
        angle=0]{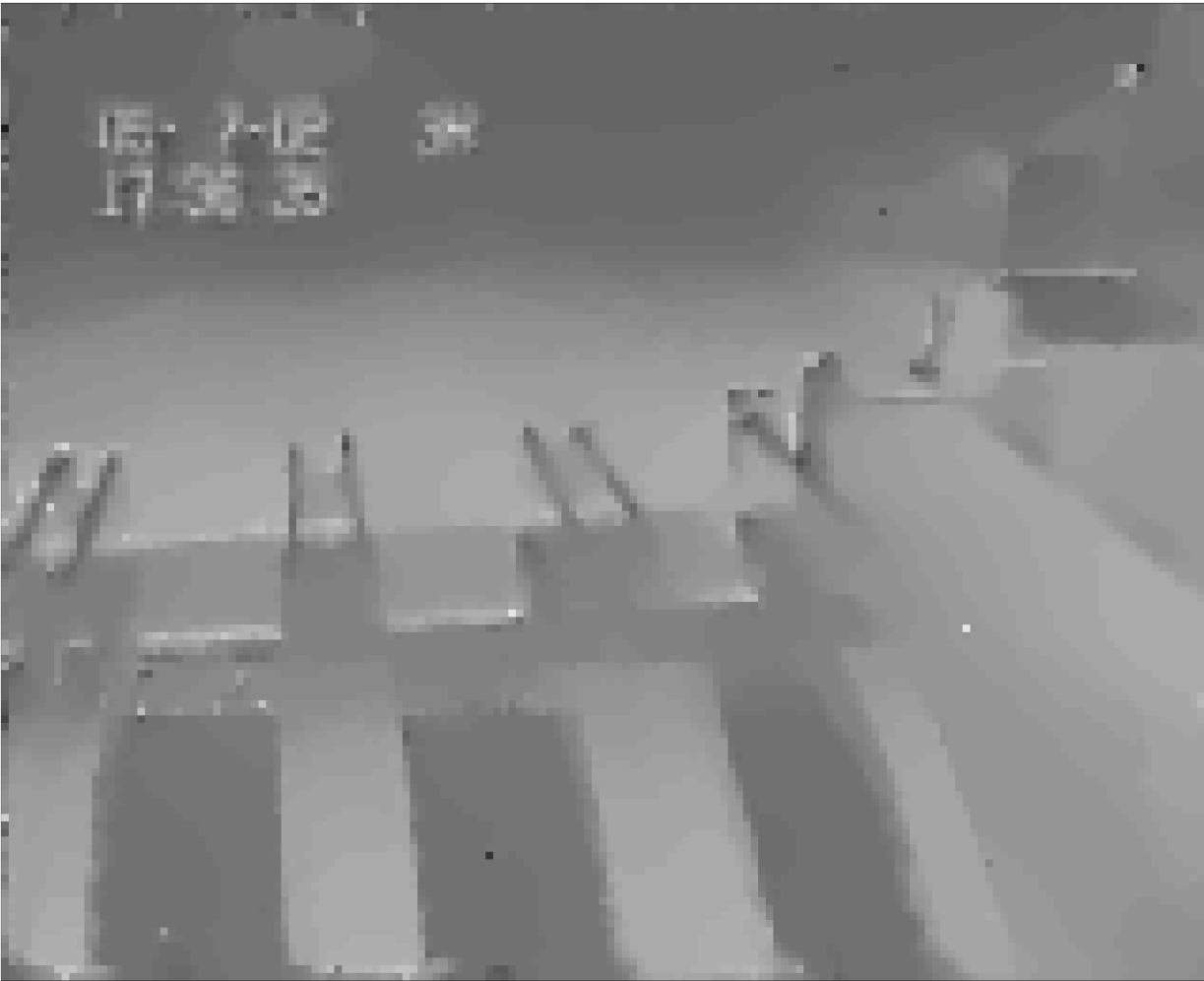}
    }
    \subfloat
    {
      \includegraphics[width=0.18\linewidth,
        angle=0]{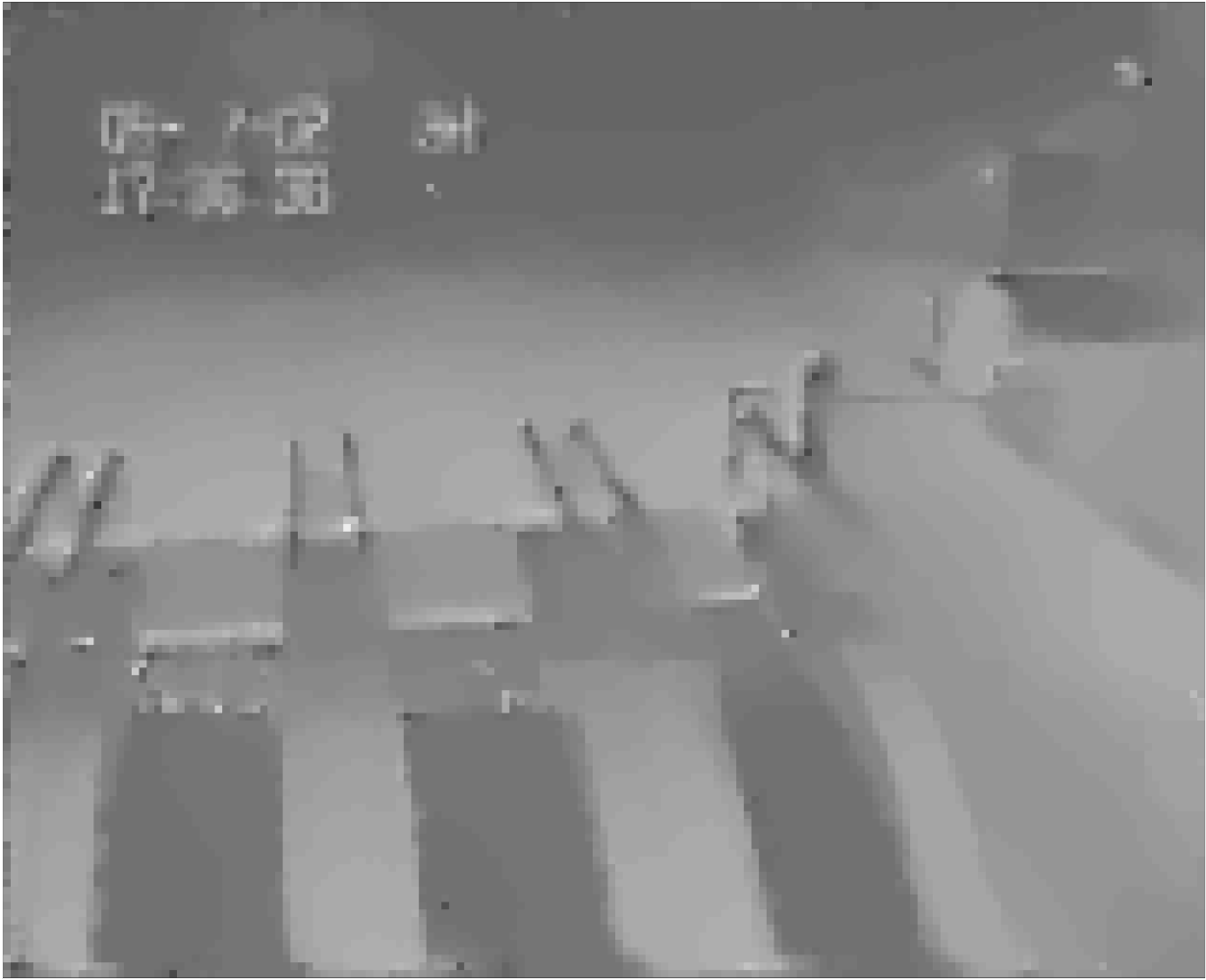}
    }
    \subfloat
    {
      \includegraphics[width=0.18\linewidth,
        angle=0]{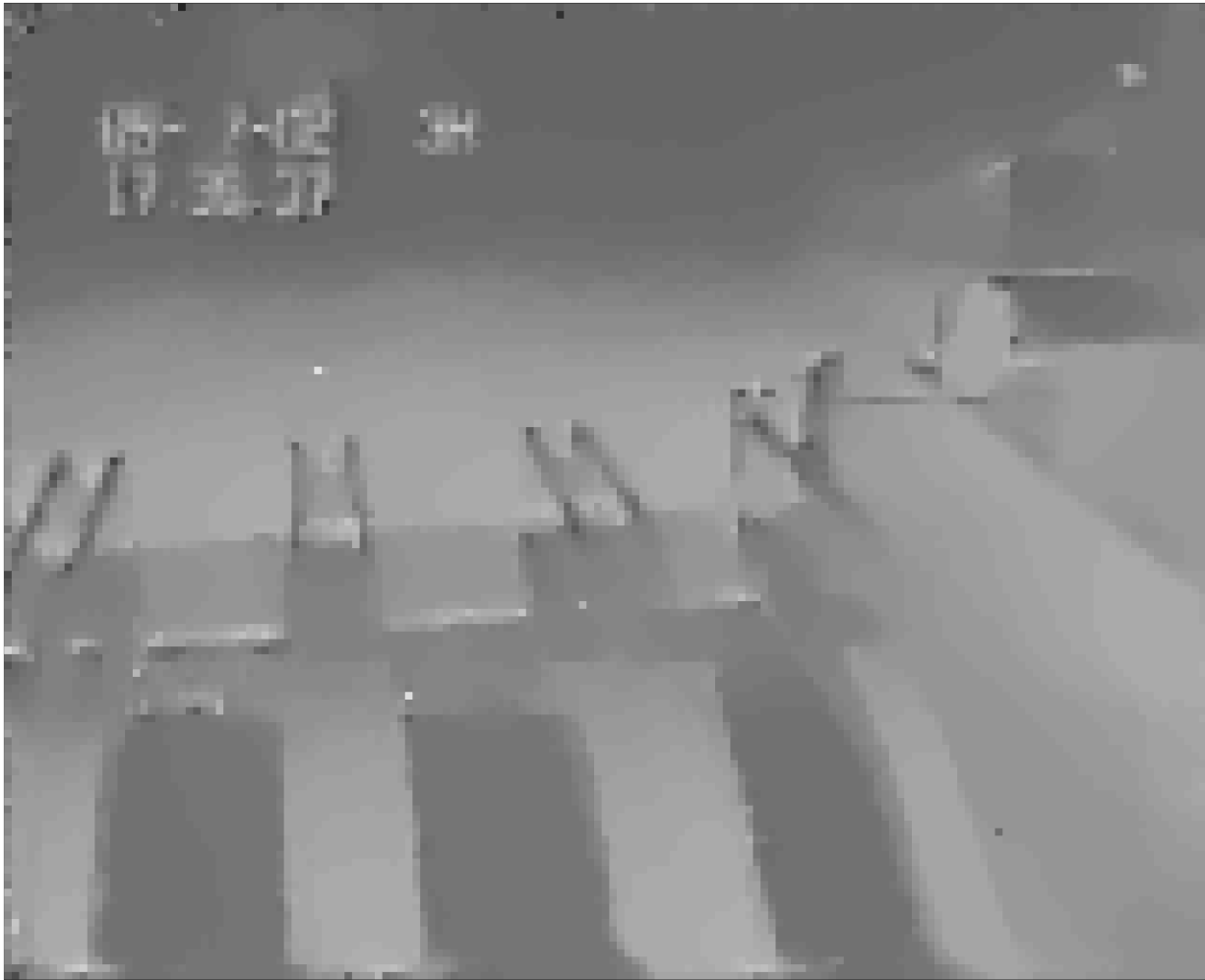}
    }
    \subfloat
    {
      \includegraphics[width=0.18\linewidth,
        angle=0]{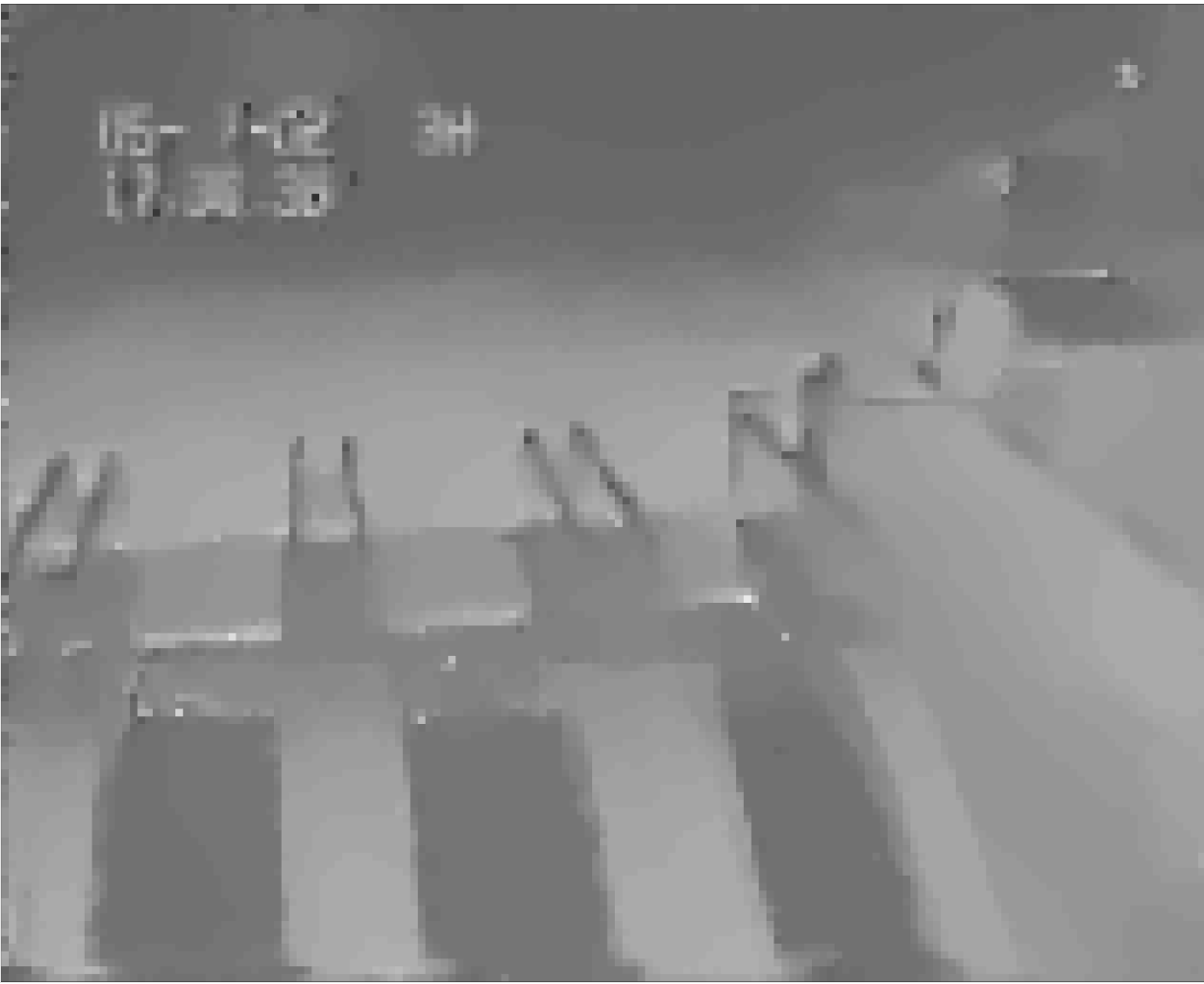}
    }
    \subfloat
    {
      \includegraphics[width=0.18\linewidth,
        angle=0]{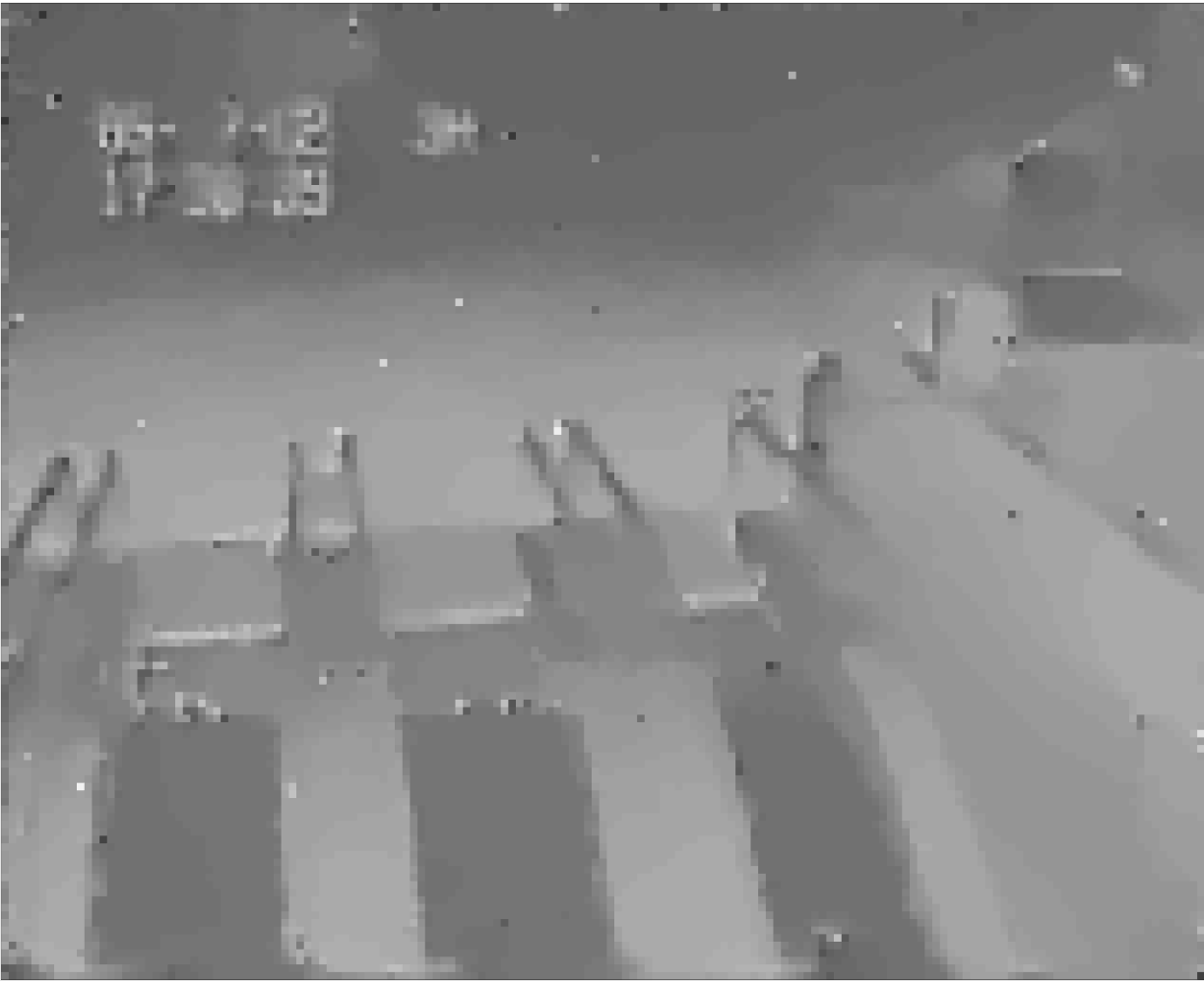}
    }
  \end{minipage}
  \caption{Contrast of initial noisy monitoring video and its denoised copy, by using proposed multidimensional TV-Stokes model. Row one and row three are sequences from initial video. Row two and row four are sequences from denoised results.}
  \label{fig:escalator}
\end{figure}

\section*{Acknowledgments}
We thank Dr. Andreas Langer for interesting discussion and for providing us with the 3d data.

\bibliographystyle{siamplain}
\bibliography{refs}

\end{document}